\documentclass[graybox]{svmult}

\usepackage{type1cm}        % activate if the above 3 fonts are
\usepackage{makeidx}         % allows index generation
\usepackage{graphicx}        % standard LaTeX graphics tool
\usepackage{multicol}        % used for the two-column index
\usepackage[bottom]{footmisc}% places footnotes at page bottom

\usepackage{newtxtext}       %
\usepackage{newtxmath}       % selects Times Roman as basic font

\usepackage[english]{babel}
\usepackage[babel=true]{microtype}
\usepackage{amsmath}

\usepackage{algpseudocode}

\usepackage{hyperref}
\usepackage{cleveref}

\theoremstyle{theorem}
\newtheorem{thm}{Theorem}
\newtheorem{dfn}[thm]{Definition}
\newtheorem{crl}[thm]{Corollary}
\newtheorem{lem}[thm]{Lemma}

\crefname{equation}{}{}
\Crefname{equation}{Equation}{Equations}
\crefname{figure}{Fig.}{Figs.}
\Crefname{figure}{Figure}{Figures}
\crefname{table}{Table}{Tables}
\Crefname{algoritheorem}{Algoritheorem}{Algoritheorems}
\crefname{section}{Sect.}{Sects.}
\Crefname{section}{Section}{Sections}

\crefname{thm}{Theorem}{Theorems}
\Crefname{thm}{Theorem}{Theorems}
\crefname{corollary}{Corollary}{Corollaries}
\Crefname{corollary}{Corollary}{Corollaries}
\crefname{lemma}{Lemma}{Lemma}
\Crefname{lemma}{Lemma}{Lemma}
\crefname{dfn}{Definition}{Definitions}
\Crefname{dfn}{Definition}{Definitions}
\crefname{remark}{Remark}{Remarks}
\Crefname{remark}{Remark}{Remarks}

\newcommand{\sepfigure}{}
\newcommand{\R}{\mathbb R}
\newcommand{\Z}{\mathbb Z}

\newcommand{\al}{\alpha}

\newcommand{\ga}{\gamma}
\newcommand{\Ga}{\Gamma}
\newcommand{\la}{\lambda}

\newcommand{\om}{\omega}
\newcommand{\bd}{\partial}

\newcommand{\at}{\mathrm{AT}}

\newcommand{\cs}{\mathrm{CS}}
\newcommand{\pr}{\mathrm{pr}}
\newcommand{\area}{\mathrm{area}}

\newcommand{\lk}{\mathrm{lk}}

\newcommand{\sign}{\mathrm{sign}}

\newcommand{\ora}{\overrightarrow}
\newcommand{\bs}{\hfill $\blacksquare$}

\newcommand{\vect}[2]{ \left( \begin{array}{c} 
 #1 \\ #2 \end{array} \right)}
\newcommand{\mat}[4]{ \left( \begin{array}{cc} 
 #1 & #2 \\ #3 & #4 \end{array} \right)}

\makeindex             % used for the subject index
\pagecolor{white}
\begin{document}

\title*{A formula for the linking number in terms of isometry invariants of straight line segments}
\titlerunning{A formula for the linking number in terms of isometry invariants}
\author{Matt Bright, Olga Anosova, Vitaliy Kurlin \newline 
Computer Science and Materials Innovation Factory, University of Liverpool \newline
sgmbrigh@liverpool.ac.uk,oanosova@liv.ac.uk, vitaliy.kurlin@gmail.com}
 \authorrunning{Bright, M., Anosova, O., Kurlin V.} 
%\institute{%
%Matt Bright, Olga Anosova, Vitaliy Kurlin \at{} University of Liverpool \email{sgmbrigh@liverpool.ac.uk,vitaliy.kurlin@gmail.com} 
% problemmas with special characters
\maketitle

\abstract{In 1833 Gauss defined the \emph{linking number} of two disjoint curves in 3-space. For open curves this double integral over the parameterised curves is real-valued and invariant modulo rigid motions or isometries that preserve distances between points, and has been recently used in the elucidation of molecular structures. In 1976 Banchoff geometrically interpreted the linking number between two line segments. An explicit analytic formula based on this interpretation was given in 2000 without proof in terms of six isometry invariants: the distance and angle between the segments and four coordinates specifying their relative positions. We give a detailed proof of this formula and describe its asymptotic behaviour that wasn't previously studied.}

%1================
\section{The Gauss integral for the linking number of curves}
\label{sec:intro}

This extended version of the conference paper \cite{bright2020proof} includes all previously skipped proofs.
For any vectors $\vec u,\vec v, \vec w\in\R^3$, the \emph{triple} product is $(\vec u,\vec v,\vec w)=(\vec u\times \vec v)\cdot\vec w$.
  
\begin{dfn}[Gauss integral for the linking number]
\label{definition:Gauss_integral}
For piecewise-smooth curves $\ga_1,\ga_2:[0,1]\to\R^3$, the \emph{linking number} can be defined as the Gauss integral \cite{gauss1833integral}
$$\lk(\ga_1,\ga_2)=\dfrac{1}{4\pi}\int\limits_0^1\int\limits_0^1 \dfrac{(\dot{\ga}_1(t),\dot{\ga}_2(s),\ga_1(t)-\ga_2(s))}{|\ga_1(t)-\ga_2(s)|^3}dt ds,
\leqno{(\ref{definition:Gauss_integral})}$$
where $\dot{\ga}_1(t),\dot{\ga}_2(s)$ are the vector derivatives of the 1-variable functions $\ga_1(t),\ga_2(s)$.
\end{dfn}

The formula in \Cref{definition:Gauss_integral} gives an integer number for any closed disjoint curves $\ga_1,\ga_2$ due to its interpretation as the \emph{degree} of the \emph{Gauss map} $\Ga(t,s)=\dfrac{\ga_1(t)-\ga_2(s)}{|\ga_1(t)-\ga_2(s)|}:S^1\times S^1\to S^2$, i.e.
$\deg\Ga=\dfrac{\area(\Ga(S^1\times S^1))}{\area(S^2)}$, where the area of the unit sphere is $\area(S^2)=4\pi$.
This integer degree is  the \emph{linking number} of the 2-component link $\ga_1\sqcup\ga_2\subset\R^3$ formed by the two closed curves.
Invariance modulo continuous deformation of $\R^3$  follows easily for closed curves - indeed, the function under the Gauss integral in (\ref{definition:Gauss_integral}), and hence the integral itself, varies continuously under perturbations of the curves $\ga_1,\ga_2$. This should keep any integer value constant.
\medskip

For open curves $\ga_1,\ga_2$, the Gauss integral gives a real but not necessarily integral value, which remains invariant under rigid motions or orientation-preserving isometries (see~\Cref{theorem:lk_invariance}).
In $\R^3$ with the Euclidean metric isometries consist of rotations, translations and reflections. 
Isometry invariance of the real-valued linking number for open curves has found applications in the study of molecules \cite{ahmad2020characterization}.
\medskip

Any smooth curve can be well-approximated by a polygonal line, so the computation of the linking number reduces to a sum over pairs of straight line segments $L_1, L_2$. 
In 1976 Banchoff \cite{banchoff1976self} has expressed the linking number $\lk(L_1,L_2)$ in terms of the endpoints of each segment, see details of this and other past work in~\Cref{sec:past}.
\medskip

In 2000 Klenin and Langowski \cite{klenin2000computation} proposed a formula  for the linking number $\lk(L_1,L_2)$ of two straight line segments in terms of six isometry invariants of $L_1,L_2$, referring to a previous paper~\cite{Vologodskii1974}, which used the formula without any detailed proof. 
The paper~\cite{klenin2000computation} also skipped all details of the invariant-based formula's derivation. 
\medskip

The usefulness of the invariant-based formula can be seen by considering the analogy with the simpler concept of the scalar (dot) product of vectors.
The algebraic or \emph{coordinate-based} formula expresses the scalar product of two vectors $\vec u=(x_1,y_1,z_1)$ and $\vec v=(x_2,y_2,z_2)$ as $\vec u\cdot\vec v=x_1x_2 +y_1y_2+z_1z_2$, which in turn depend on the co-ordinates of their endpoints. However, the scalar product for  high-dimensional vectors $\vec u,\vec v\in \R^n$ can also expressed in terms of only 3 parameters $\vec u\cdot\vec v=|\vec u|\cdot|\vec v|\cos\angle(\vec u,\vec v)$.
The two lengths $|\vec u|,|\vec v|$ and the angle $\angle(\vec u,\vec v)$ are isometry invariants of the vectors $\vec u,\vec v$.
This second geometric or \emph{invariant-based} formula makes it clear that $\vec u\cdot\vec v$ is an isometry invariant, while it is harder to show that $\vec u\cdot\vec v=x_1x_2 +y_1y_2+z_1z_2$ is invariant under rotations. It also provides other geometric insights that are hard to extract from the coordinate-based formula - for example, $\vec u\cdot\vec v$ oscillates as a cosine wave when the lengths $|\vec u|,|\vec v|$ are fixed, but the angle $\angle(\vec u,\vec v)$ is varying.
\medskip

In this paper, we provide a detailed proof of the invariant-based formula for the linking number in~\Cref{theorem:lk_arctan} and new corollaries in~\Cref{sec:behaviour} formally investigating the asymptotic behaviour of the linking number, which wasn't previously studied.
\medskip

Our own interest in the asymptotic behaviour is motivated by the \emph{periodic linking number} by Panagiotou \cite{panagiotou2015linking} as an invariant of crystalline networks \cite{cui2019mining} that are infinitely periodic in three directions, by calculating the infinite sum of the linking number between one line segment and all translated copies of another such segment. In \cite{panagiotou2015linking} there is a complex proof that this sum is convergent for a cubical lattice.
The convergence of the periodic linking numbers remains open for arbitrary lattices. 
%could be simplified and improved by a new asymptotic analysis of the closed form. 

%2================
\section{Outline of the invariant-based formula and consequences}
\label{sec:outline}

Folklore Theorem~\ref{theorem:lk_invariance} lists key properties of $\lk(\ga_1,\ga_2)$, which will be used later.

\begin{thm}[properties of the linking number]
\label{theorem:lk_invariance}
The linking number defined by the Gauss integral in~\Cref{definition:Gauss_integral} for smooth curves $\ga_1,\ga_2$ has the following properties:
\smallskip

\noindent
(\ref{theorem:lk_invariance}a)
the linking number is symmetric: $\lk(\ga_1,\ga_2)=\lk(\ga_2,\ga_1)$;
\smallskip

\noindent
(\ref{theorem:lk_invariance}b)
$\lk(\ga_1,\ga_2)=0$ for any curves $\ga_1,\ga_2$ that belong to the same plane;
\smallskip

\noindent
(\ref{theorem:lk_invariance}c)
$\lk(\ga_1,\ga_2)$ is independent of orientation-preserving parameterisations of the open curves $\ga_1,\ga_2$ with fixed endpoints;
\smallskip

\noindent
(\ref{theorem:lk_invariance}d)
$\lk(-\ga_1,\ga_2)=-\lk(\ga_1,\ga_2)$, where $-\ga_1$ has the reversed orientation of $\ga_1$;
\smallskip

\noindent
(\ref{theorem:lk_invariance}e)
the linking number $\lk(\ga_1,\ga_2)$ is invariant under any scaling $\vec v\to \la\vec v$ for $\la>0$;
\smallskip

\noindent
(\ref{theorem:lk_invariance}f)
$\lk(\ga_1,\ga_2)$ is multiplied by $\det M$ under any orthogonal map $\vec v\mapsto M\vec v$. %, $M\in O(\R^3)$.
\end{thm}

%These are frequently stated without proof in the literature - we give some outline proofs here:

\begin{proof}
(\ref{theorem:lk_invariance}a) 
We note that the Euclidean distance is symmetric, and that since the triple product is anti-symmetric and $\ga_2(s) - \ga_1(t) = -(\ga_1(t) - \ga_2(s)$, the symmetry follows from
\begin{align*}
(\dot{\ga}_1(s),\dot{\ga}_2(t),\ga_2(s)-\ga_2(t)) &= -(\dot{\ga}_1(t),\dot{\ga}_2(s),  \ga_2(s)-\ga_1(t)) \\
 & = -(\dot{\ga}_1(t),\dot{\ga}_2(s),  -(\ga_1(t)-\ga_2(s)) ) \\
 &=  (\dot{\ga}_1(t),\dot{\ga}_2(s),  \ga_1(t)-\ga_2(s) ).
\end{align*}

\noindent
(\ref{theorem:lk_invariance}b) is obvious from the coplanarity of the normal vectors.
\smallskip

\noindent
(\ref{theorem:lk_invariance}c) is simply a consequence of the path-independence of the integrals over $s,t$.
\smallskip

\noindent
(\ref{theorem:lk_invariance}d) follows from $\dot{\gamma}_1(-t) = -\dot{\gamma}(t)$ since the reverse orientation of $\gamma_1(t)$ is $\gamma_1(-t)$.
\smallskip
 
\noindent
(\ref{theorem:lk_invariance}e)
Any scaling $\vec{v} \mapsto \lambda\vec{v}$ will result in a change of parameterisation $\gamma_i(t) \mapsto \gamma_i(\lambda t) = \lambda(\gamma_I)(t)$. 
Since $\dot{\lambda \gamma}_i(t) = \lambda \dot{\gamma}_i(t)$, the result follows below
 
\begin{align*}
 \lk(\ga_1(\lambda t),\ga_2(\lambda s)) & =\dfrac{1}{4\pi}\int\limits_0^1\int\limits_0^1 \dfrac{(\lambda \dot{\ga}_1(t),\lambda\dot{\ga}_2(s),\lambda(\ga_1(t)-\ga_2(s)))}{|\lambda(\ga_1(t)-\ga_2(s))|^3}dt ds\\
& =\dfrac{1}{4\pi} \int\limits_0^1\int\limits_0^1  \dfrac{\lambda^3(\dot{\ga}_1(t),\dot{\ga}_2(s),\ga_1(t)-\ga_2(s))}{\lambda^3 |\ga_1(t)-\ga_2(s)|^3}dt ds\\
&= \lk(\ga_1(t), \ga_2(s)),
 \end{align*}
\smallskip
 
\noindent
(\ref{theorem:lk_invariance}f) 
For an orthogonal transformation $M$, we have $M\vec{u} \times M\vec{v} = (\det M)  M (\vec{u} \times \vec{v})$, while $M\vec{u}\cdot M\vec{v} = \vec{u}\cdot\vec{v}$. 
Therefore $|M\vec{v} - M\vec{u}| = |\vec{v} - \vec{u}|$, $(M\vec{u}, M\vec{v}, M\vec{w}) = \det M(\vec{u}, \vec{v}, \vec{w})$
and $\lk(M\gamma_1, M \gamma_2) = (\det{M})  (\lk (\gamma_1, \gamma_2))$, so the linking number is multiplied by $\det{M}$.
\bs 
\end{proof}

Our main~\Cref{theorem:lk_arctan} will prove an analytic formula for the linking number of any line segments $L_1,L_2$ in terms of 6 isometry invariants of $L_1,L_2$, which are introduced in~\Cref{lemma:segments_parameters}.
Simpler~\Cref{cor:lk_orthogonal_at0} expresses $\lk(L_1,L_2)$ for any \emph{simple} orthogonal oriented segments $L_1,L_2$ defined by their lengths $l_1,l_2>0$ and initial endpoints $O_1,O_2$, respectively, with the Euclidean distance $d(O_1,O_2)=d>0$, so that $\vec L_1,\vec L_2,\ora{O_1O_2}$ form a positively oriented orthogonal basis whose signed volume $(\vec L_1,\vec L_2,\ora{O_1O_2})=l_1l_2d$ is the product of the lengths, see the first picture in Fig.~\ref{fig:segments_parameters}.

\begin{crl}[linking number for simple orthogonal segments]
\label{cor:lk_orthogonal_at0}
For any simple orthogonal oriented line segments $L_1,L_2\subset\R^3$ with lengths $l_1,l_2$ and a distance $d$ as defined above, the linking number is 
$\lk(L_1,L_2)=-\dfrac{1}{4\pi}\arctan\left(\dfrac{l_1 l_2}{d\sqrt{l_1^2+l_2^2+d^2}}\right)$.
%\bs
%\leqno{(\Cref{theorem:lk_orthogonal_at0})}
\end{crl}

The above expression is a special case of general formula (\ref{theorem:lk_arctan}) for $a_1=a_2=0$ and $\al=\dfrac{\pi}{2}$. 
%Both formulae are invariant under the uniform scaling of $\R^3$ by $\la$, which agrees with~\Cref{theorem:lk_invariance}.
If $l_1=l_2=l$, the linking number in~\Cref{cor:lk_orthogonal_at0} becomes $\lk(L_1,L_2)=-\dfrac{1}{4\pi}\arctan\dfrac{l^2}{d\sqrt{2l^2+d^2}}$.
%If $l_1=l_2=\dfrac{d}{2}$, then $\lk(L_1,L_2)=\dfrac{1}{2\pi}\left(\arcsin\dfrac{1}{\sqrt{2.5}}-\dfrac{\pi}{4}\right)\approx -0.016$.
If $l_1=l_2=d$, then $\lk(L_1,L_2)=-\dfrac{1}{4\pi}\arctan\dfrac{1}{\sqrt{3}}=-\dfrac{1}{24}$. %\approx -0.0417$.
\medskip

\Cref{cor:lk_orthogonal_at0} implies that the linking number is in the range $(-\frac{1}{8},0)$ for any simple orthogonal segments with $d>0$, which wasn't obvious from~\Cref{definition:Gauss_integral}. %In~\Cref{fig:segments_parameters} the choice of orientation generates a negative crossing for $d>0$. 
If $L_1,L_2$ move away from each other, then
$\lim\limits_{d\to+\infty} \lk(L_1,L_2)=-\dfrac{1}{4\pi}\arctan 0=0$.
Alternatively, if segments with $l_1=l_2=l$ become infinitely short, the limit is again zero: $\lim\limits_{l\to 0}\lk(L_1,L_2)=0$ for any fixed  $d$.
The limit $\lim\limits_{x\to+\infty}\arctan x=\dfrac{\pi}{2}$ implies that if  segments with $l_1=l_2=l$ become infinitely long for a fixed distance $d$, $\lim\limits_{l\to +\infty}\lk(L_1,L_2)=-\dfrac{1}{4\pi}\arctan\dfrac{l^2}{d\sqrt{2l^2+d^2}}=-\dfrac{1}{8}$.
If we push segments $L_1,L_2$ of fixed (possibly different) lengths $l_1,l_2$ towards each other, the same limit similarly emerges:
$\lim\limits_{d\to 0}\lk(L_1,L_2)=-\dfrac{1}{8}$.
See more general corollaries in section~\ref{sec:behaviour}.

%3================
\section{Past results about the Gauss integral for the linking number}
\label{sec:past}

The survey \cite{ricca2011gauss}  reviews the history of the Gauss integral, its use in Maxwell's description of electromagnetic fields \cite{maxwell1873treatise}, and its interpretation as the degree of a map from the torus to the sphere.
In classical knot theory $\lk(\ga_1,\ga_2)$ is a topological invariant of a link consisting of closed curves $\ga_1\sqcup\ga_2$, whose equivalence relation is ambient isotopy. This relation is too flexible for open curves which can be isotopically unwound, and hence doesn't preserve the Gauss integral for open curves $\ga_1,\ga_2$.
\medskip

Computing the value of the Gauss integral directly from the parametric equation of two generic curves is only possible by approximation, but this problem is simplified when we consider simply straight lines. 
The first form of the linking number between two straight line segments in terms of their geometry is described by Banchoff \cite{banchoff1976self}. Banchoff considers the projection of segments on to a plane orthogonal to some vector $\xi \in S^2$. The Gauss integral is interpreted as the fraction of the unit sphere covered by those directions of $\xi$ for which the projection will have a crossing. 
\medskip

This interpretation was the foundation of a closed form  developed by Arai \cite{arai2013rigorous}, using van Oosterom and Strackee's closed formula for the solid angle subtended by a tetrahedron given by the origin of a sphere and three points on its surface. 
An efficient implementation of the solid angle approach to the linking number is discussed in \cite{bertolazzi2019efficient}.
\medskip

An alternative calculation for this solid angle is given in \cite{panagiotou2020knot} as a starting point for calculating further invariants of open entangled curves. 
This form does not employ geometric invariants, but was used in~\cite{klenin2000computation} to claim a formula (without a proof) similar to~\Cref{theorem:lk_arctan}, which is proved in this paper with more corollaries in section~\ref{sec:behaviour}. 

\sepfigure

\begin{figure}[h]
\includegraphics[width=\textwidth]{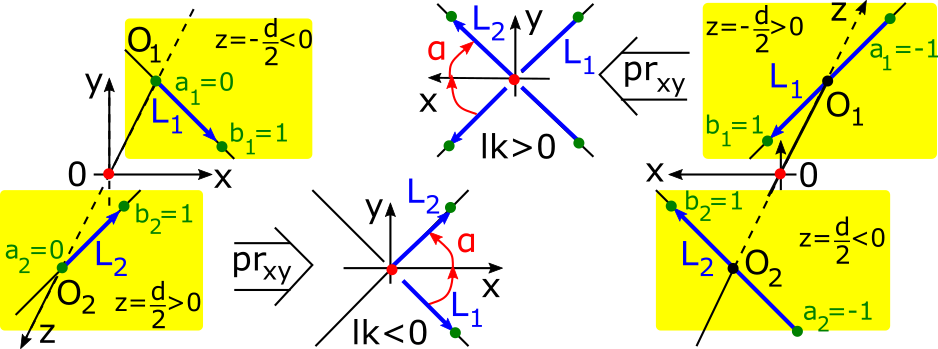}
\label{fig:segments_parameters}
\caption{Each line segment $L_i$ is in the plane $\{z=(-1)^i \frac{d}{2}\}$, $i=1,2$.
\textbf{Left}: signed distance $d>0$, the endpoint coordinates $a_1=0$, $b_1=1$ and $a_2=0$, $b_2=1$, the lengths $l_1=l_2=1$.
\textbf{Right}: signed distance $d<0$, the endpoint coordinates $a_1=-1$, $b_1=1$ and $a_2=-1$, $b_2=1$, so $l_1=l_2=2$.
In both middle pictures $\al=\frac{\pi}{2}$ is the angle from $\pr_{xy}(L_1)$ to $\pr_{xy}(L_2)$ with $x$-axis as the bisector.}
\end{figure}

\sepfigure

%4================
\section{Six isometry invariants of skew line segments in 3-space}
\label{sec:parameterization}

This section introduces six isometry invariants, which uniquely determine positions of any line segments $L_1,L_2\subset\R^3$ modulo isometries of $\R^3$, see ~\Cref{lemma:segments_parameters}.
\medskip

It suffices to consider only \emph{skew} line segments that do not belong to the same 2-dimensional plane.
If $L_1,L_2$ are in the same plane $\Pi$, for example if they are parallel, then $\dot L_1(t)\times \dot L_2(s)$ is orthogonal to any vector $L_1(t)-L_2(s)$ in the plane $\Pi$, hence $\lk(L_1,L_2)=0$.
We denote by $\bar L_1,\bar L_2\subset\R^3$ the infinite oriented lines through the given line segments $L_1,L_2$, respectively.
In a plane with fixed coordinates $x,y$, all angles are  measured anticlockwise from the positive $x$-axis.

\begin{dfn}[invariants of line segments]
\label{definition:segments_parameters}
Let $\al\in[0,\pi]$ be the angle between oriented line segments $L_1,L_2\subset\R^3$.
Assuming that $L_1,L_2$ are not parallel, there is a unique pair of parallel planes $\Pi_i$, $i=1,2$, each containing the infinite line $\bar L_i$ through the line segment $L_i$.
We choose orthogonal coordinates $x,y,z$ in $\R^3$ so that  
\medskip

\noindent
(\ref{definition:segments_parameters}a)
the horizontal plane $\{z=0\}$ is in the middle between $\Pi_1,\Pi_2$, see~Fig.~\ref{fig:segments_parameters};
\medskip

\noindent
(\ref{definition:segments_parameters}b)
$(0,0,0)$ is the intersection of the projections $\pr_{xy}(\bar L_1),\pr_{xy}(\bar L_2)$ to $\{z=0\}$;
\medskip

\noindent
(\ref{definition:segments_parameters}c)
the $x$-axis bisects the angle $\al$ from $\pr_{xy}(\bar L_1)$ to $\pr_{xy}(\bar L_2)$, the $y$-axis is chosen so that $\al$ is anticlockwisely measured from the $x$-axis to the $y$-axis in $\{z=0\}$;
\medskip

\noindent
(\ref{definition:segments_parameters}d)
the $z$-axis is chosen so that $x,y,z$ are oriented in the right hand way, then $d$ is the \emph{signed} distance from $\Pi_1$ to $\Pi_2$; the distance $d$ is negative if the vector $\ora{O_1O_2}$ is opposite to the positively oriented $z$-axis in Fig.~\ref{fig:segments_parameters}.
\medskip

\noindent
Let $a_i,b_i$ be the coordinates of the initial and final endpoints of the segments $L_i$ in the infinite line $\bar L_i$ whose origin is $O_i=\Pi_i\cap(z\text{-axis})=(0,0,(-1)^i\frac{d}{2})$, $i=1,2$. 
%see~\Cref{fig:segments_parameters}.
\end{dfn}

The case of segments $L_1,L_2$ lying in the same plane $\Pi\subset\R^3$ can be formally covered by~\Cref{definition:segments_parameters} if we allow the signed distance $d$ from $\Pi_1$ to $\Pi_2$ to be 0.

\begin{lem}[parameterisation]
\label{lemma:segments_parameters}
Any oriented line segments $L_1,L_2\subset\R^3$ are uniquely determined modulo a rigid motion by their isometry invariants $\al\in[0,\pi]$ and $d$, $a_1$, $b_1$, $a_2$, $b_2\in\R$ from~\Cref{definition:segments_parameters}.
For $l_i=b_i-a_i$, $i=1,2$, each line segment $L_i$ is 
$$L_i(t) = \Big(\; 
(a_i+l_i t)\cos\frac{\al}{2},\; 
(-1)^i(a_i+l_i t)\sin\frac{\al}{2},\; 
(-1)^i \frac{d}{2} \;\Big), \; t\in[0,1].
\leqno{(\ref{lemma:segments_parameters})}$$
\end{lem}

\begin{proof}
Any line segments $L_1,L_2\subset\R^3$ that are not in the same plane are contained in distinct parallel planes.
For $i=1,2$, the plane $\Pi_i$ is spanned by $L_i$ and the line parallel to $L_{3-i}$ and passing through an endpoint of $L_i$.  
Let $L'_i$ be the orthogonal projection of the line segment $L_i$ to the plane $\Pi_{3-i}$.
The non-parallel lines through the segments $L_i$ and $L'_{3-i}$ in the plane $\Pi_i$ intersect at a point, say $O_i$.
Then the line segment $O_1 O_2$ is orthogonal to both planes $\Pi_i$, hence to both $L_i$ for $i=1,2$.
\medskip

By~\Cref{theorem:lk_invariance}, to compute $\lk(L_1,L_2)$, one can apply a rigid motion to move the mid-point of the line segment $O_1O_2$ to the origin $O=(0,0,0)\in\R^2$ and make $O_1 O_2$ vertical, i.e. lying within the $z$-axis.
The signed distance $d$ can be defined as the difference between the coordinates of $O_2=\Pi_2\cap(z\text{-axis})$ and $O_1=\Pi_1\cap(z\text{-axis})$ along the $z$-axis.
Then $L_i$ lies in the horizontal plane $\Pi_i=\{z=(-1)^i \frac{d}{2}\}$, $i=1,2$.
\medskip

An extra rotation around the $z$-axis  guarantees that the $x$-axis in the horizontal plane $\Pi=\{z=0\}$ is the bisector of the angle $\al\in[0,\pi]$ from $\pr_{xy}(\bar L_1)$ to $\pr_{xy}(\bar L_2)$, where $\pr_{xy}:\R^3\to\Pi$ is the orthogonal projection. 
Then the infinite lines $\bar L_i$ through $L_i$ have the parametric form $(x,y,z)=(t\cos\frac{\al}{2},(-1)^i t\sin\frac{\al}{2},(-1)^i\frac{d}{2})$ with $s\in\R$.
\medskip

The point $O_i$ can be considered as the origin of the oriented infinite line $\bar L_i$.
Let the line segment $L_i$ have a length $l_i>0$ and its initial point have the coordinate $a_i\in\R$ in the oriented line $\bar L_i$.
Then the final endpoint of  $L_i$ has the coordinate $b_i=a_i+l_i$.
To cover only the segment $L_i$, the parameter $t$ should be replaced by $a_i+l_it$, $t\in[0,1]$.
\bs
\end{proof}

If $t\in\R$ in~\Cref{lemma:segments_parameters}, the corresponding point $L_i(t)$ moves along the line $\bar L_i$.
%The left hand side picture in Fig.~\Cref{fig:segments_parameters} has the line segments $L_1$ from $(-\frac{1}{\sqrt{2}},\frac{1}{\sqrt{2}},-1)$ to $(\frac{1}{\sqrt{2}},-\frac{1}{\sqrt{2}},-1)$ and $L_2$ from $(-\frac{1}{\sqrt{2}},-\frac{1}{\sqrt{2}},1)$ to $(\frac{1}{\sqrt{2}},\frac{1}{\sqrt{2}},1)$, whose $xy$-projections form the angle $\al=\frac{\pi}{4}$ with the $x$-axis.
\medskip

\begin{lem}[formulae for invariants] 
\label{lemma:formulae_parameters}
Let $L_1,L_2\subset\R^3$ be any skewed oriented line segments given by their initial and final endpoints $A_i,B_i\in\R^3$ so that $\vec L_i=\ora{A_iB_i}$, $i=1,2$.
Then the isometry invariants of $L_1,L_2$ in~\Cref{lemma:segments_parameters} are computed as follows:
\medskip

\noindent
the lengths $l_i=|\ora{A_iB_i}|$, the signed distance $d=\dfrac{[\vec L_1,\vec L_2,\ora{A_1 A_2}]}{|\vec L_1\times\vec L_2|}$, the angle $\al=\arccos\dfrac{\vec L_1\cdot\vec L_2}{l_1l_2}$,
$a_1=\left(\dfrac{\vec L_2}{l_2}\cos\al-\dfrac{\vec L_1}{l_1}\right)\cdot\dfrac{\ora{A_1 A_2}}{\sin^2\al}$,
$\; a_2=\left(\dfrac{\vec L_2}{l_2}-\dfrac{\vec L_1}{l_1}\cos\al\right)\cdot\dfrac{\ora{A_1 A_2}}{\sin^2\al}$,
$\; b_i=a_i+l_i$, $i=1,2$.
\end{lem}

\begin{proof}
The vectors along the segments are $\vec L_i=\vec v_i-\vec u_i$, hence the lengths are $l_i=|\vec L_i|=|\ora{A_iB_i}|$, $i=1,2$.
The angle $\al\in[0,\pi]$ between $\vec L_1,\vec L_2$ can be found from the scalar product $\vec L_1\cdot\vec L_2=|\vec L_1|\cdot|\vec L_2|\cos\al$ as $\al=\arccos\dfrac{\vec L_1\cdot\vec L_2}{l_1l_2}$, because the function $\arccos x:[-1,1]\to[0,\pi]$ is bijective. 
\medskip

Since the vectors $\vec L_1,\vec L_2$ are not proportional to each other, the normalised vector product $\vec e_3=\dfrac{\vec L_1\times\vec L_2}{|\vec L_1\times\vec L_2|}$ is well-defined and orthogonal to both $\vec L_1,\vec L_2$.
Then $\vec e_1=\dfrac{\vec L_1}{|\vec L_1|}$, $\vec e_2=\dfrac{\vec L_2}{|\vec L_2|}$ and $\vec e_3$ have lengths 1 and form a linear basis of $\R^3$, where the last vector is orthogonal to the first two.
\medskip

Let $O$ be any fixed point of $\R^3$, which can be assume to be
the origin $(0,0,0)$ in the coordinates of~\Cref{lemma:segments_parameters}, though its position relative to $\ora{A_iB_i}$ is not yet determined.
First we express the points $O_i=(0,0,(-1)^i\frac{d}{2})\in\bar L_i$ from Fig.~\ref{fig:segments_parameters} in terms of given vectors $\ora{A_iB_i}$.
If the initial endpoint $A_i$ has a coordinate $a_i$ in the line $\bar L_i$ through $L_i$, then $\ora{O_i A_i}=a_i\vec e_i$ and
$$\ora{O_1 O_2}=\ora{O O_2}-\ora{O O_1}=(\ora{O A_2}-\ora{O_2 A_2})-(\ora{O A_1}-\ora{O_1 A_1})=\ora{A_1 A_2}+a_1\vec e_1-a_2\vec e_2.$$
%The point $O_i$ has a vector $\vec u_i+t_i(\vec v_i-\vec u_i)$ for some $t_i\in\R$.
%If we set $\vec u_{12}=\vec u_2-\vec u_1$, then $\ora{O_1 O_2}=\vec u_{12}+t_2\vec L_2-t_1\vec L_1$.
By~\Cref{definition:segments_parameters}, $\ora{O_1 O_2}$ is orthogonal to the line $\bar L_i$ going through the vector $\vec e_i=\dfrac{\vec L_i}{|\vec L_i|}$ for $i=1,2$. 
Then the product $[\vec e_1,\vec e_2,\ora{O_1 O_2}]=(\vec e_1\times\vec e_2)\cdot\ora{O_1 O_2}$ equals $|\vec e_1\times\vec e_2|d$, where $\ora{O_1O_2}$ is in the $z$-axis, the signed distance $d$ is the $z$-coordinate of $O_2$ minus the $z$-coordinates $O_1$.
\medskip

The product $[\vec e_1,\vec e_2,\ora{O_1 O_2}]=(\vec e_1\times\vec e_2)\cdot(\ora{A_1 A_2}+a_1\vec e_1-a_2\vec e_2)=(\vec e_1\times\vec e_2)\cdot\ora{A_1 A_2}$ doesn't depend on $a_1,a_2$, because $\vec e_1\times\vec e_2$ is orthogonal to both $\vec e_1,\vec e_2$.
Hence the signed distance is $d=\dfrac{[\vec e_1,\vec e_2,\ora{A_1 A_2}]}{|\vec e_1\times\vec e_2|}=\dfrac{[\vec L_1,\vec L_2,\ora{A_1 A_2}]}{|\vec L_1\times\vec L_2|}$, which can be positive or negative, see Fig.~\ref{fig:segments_parameters}.
\medskip

It remains to find the coordinate $a_i$ of the initial endpoint of $L_i$ relative to the origin $O_i\in \bar L_i$, $i=1,2$. 
The vector $\ora{O_1O_2}=\ora{A_1 A_2}+a_1\vec e_1-a_2\vec e_2$ is orthogonal to both $\vec e_i$ if and only if the scalar products vanish: 
$\ora{O_1O_2}\cdot\vec e_i=0$.
Due to $|\vec e_1|=1=|\vec e_2|$ and $\vec e_1\cdot\vec e_2=\cos\al$, we get
$$\left\{\begin{array}{l}
\vec e_1\cdot\ora{A_1 A_2}+a_1-a_2(\vec e_1\cdot\vec e_2)=0, \\
\vec e_2\cdot\ora{A_1 A_2}+a_1(\vec e_1\cdot\vec e_2)-a_2=0, \\
\end{array}\right.\qquad
\mat{1}{-\cos\al}{\cos\al}{-1}\vect{a_1}{a_2}=-\vect{\vec e_1\cdot\ora{A_1 A_2}}{\vec e_2\cdot\ora{A_1 A_2}}.$$
The determinant of the $2\times 2$ matrix is $\cos^2\al-1=-\sin^2\al\neq 0$, because $L_1,L_2$ are not parallel.
Then 
$\vect{a_1}{a_2}=\dfrac{1}{\sin^2\al}\mat{-1}{\cos\al}{-\cos\al}{1}\vect{\vec e_1\cdot\ora{A_1 A_2}}{\vec e_2\cdot\ora{A_1 A_2}}$.
$$a_1
=\dfrac{-\vec e_1\cdot\ora{A_1 A_2}+\cos\al(\vec e_2\cdot\ora{A_1 A_2})}{\sin^2\al}
=\dfrac{(\vec e_2\cos\al-\vec e_1)\cdot\ora{A_1 A_2}}{\sin^2\al}
=\left(\dfrac{\vec L_2}{l_2}\cos\al-\dfrac{\vec L_1}{l_1}\right)\cdot\dfrac{\ora{A_1 A_2}}{\sin^2\al},$$
$$a_2
=\dfrac{\cos\al(\vec e_1\cdot\ora{A_1 A_2})-\vec e_1\cdot\ora{A_1 A_2}}{\sin^2\al}
=\dfrac{(\vec e_2-\vec e_1\cos\al)\cdot\ora{A_1 A_2}}{\sin^2\al}
=\left(\dfrac{\vec L_2}{l_2}-\dfrac{\vec L_1}{l_1}\cos\al\right)\cdot\dfrac{\ora{A_1 A_2}}{\sin^2\al}.$$
The coordinates of the final endpoints are obtained as $b_i=a_i+l_i$, $i=1,2$. 
\bs
\end{proof}

%The value $\al=\frac{\pi}{2}$ in lemmama~\Cref{lemma:segments_parameters} means that the angle between the oriented line segments $L_1,L_2$ is $2\al=\pi$, i.e. they are anti-parallel, so $\lk(L_1,L_2)=0$.
%Hence $\al\in(0,\frac{\pi}{2})\cup(\dfrac{\pi}{2},\pi)$.

\Cref{lemma:central_symmetry} guarantees that the linking number behaves symmetrically in $d$, meaning that we may confine any particular analysis to cases where $d>0$ or $d<0$. %, as appropriate. 

\begin{lem}[symmetry]
\label{lemma:central_symmetry}
Let $L_1,L_2\subset\R^3$ be parameterised as in~\Cref{lemma:segments_parameters}.
Under the central symmetry $\cs:(x,y,z)\mapsto(-x,-y,-z)$ with respect to the origin $(0,0,0)\in\R^3$, the line segments keep their invariants $\al,a_1,b_1,a_2,b_2$.
The signed distance $d$ and the linking number change their signs: $\lk(\cs(L_1),\cs(L_2))=-\lk(L_1,L_2)$.
\end{lem}

\begin{proof}[\Cref{lemma:central_symmetry}]
Under the central symmetry $\cs$, in the notation of~\Cref{lemma:formulae_parameters} the vectors $\vec L_1,\vec L_2,\ora{A_1A_2}$ change their signs.
Then the formulae for $\al,a_1,b_1,a_2,b_2$ gives the same expression, but the triple product $[\vec L_1,\vec L_2,\ora{A_1A_2}]$ and $d$ change their signs.
\medskip

Since the central symmetry $\cs$ is an orthogonal map $M$ with $\det M=-1$, the new linking number changes its sign as follows: $\lk(\cs(L_1),\cs(L_2))=\lk(\cs(L_2),\cs(L_1))=-\lk(L_1,L_2)$, where we also make use of the invariance of the linking number under exchange of the segments from~\Cref{theorem:lk_invariance}(f).
\bs
\end{proof}

%5================
\section{Invariant-based formula for the linking number of segments}
\label{sec:lk_formula}

This section proves main~\Cref{theorem:lk_arctan}, which expresses the linking number of any line segments in terms of their 6 isometry invariants from~\Cref{definition:segments_parameters}.
In 2000 Klenin and Langowski claimed a similar formula \cite{klenin2000computation}, but gave no proof, which requires substantial lemmas below.
One of their 6 invariants differs from the signed distance $d$.

\begin{thm}[invariant-based formula]
\label{theorem:lk_arctan}
For any line segments $L_1,L_2\subset\R^3$ with invariants $\al\in(0,\pi)$, $a_1,b_1,a_2,b_2,d\in\R$ from~\Cref{definition:segments_parameters}, we have
$\lk(L_1,L_2)=$
$$\dfrac{\at(a_1,b_2;d,\al)+\at(b_1,a_2;d,\al)-\at(a_1,a_2;d,\al)-\at(b_1,b_2;d,\al)}{4\pi},\eqno{(\ref{theorem:lk_arctan})}$$
where  
$\at(a,b;d,\al)=\arctan\left(\dfrac{ab\sin\al+d^2\cot\al}{d\sqrt{a^2+b^2-2ab\cos\al+d^2}}\right)$.
For $\al=0$ or $\al=\pi$, we set $\at(a,b;d,\al)=\sign(d)\dfrac{\pi}{2}$.
We also set $\lk(L_1,L_2)=0$ when $d=0$.
\end{thm}

$a^2+b^2-2ab\cos\al$ gives the squared third side of the triangle with the first two sides $a,b$ and the angle $\al$ between them, hence is always non-negative.
Also $a^2+b^2-2ab\cos\al=0$ only when the triangle degenerates for $a=\pm b$ and $\cos\al=\pm 1$.
For $\al=0$ or $\al=\pi$ when $L_1,L_2$ are parallel, $\lk(L_1,L_2)=0$ is guaranteed by $\at(a,b;d,\al)=\sign(d)\dfrac{\pi}{2}=0$ when $d=0$ holds in addition to $\al=0$ or $\al=\pi$.
%Corollary~\Cref{theorem:lk_d->0} will discuss the (dis)continuity of formula~(\Cref{theorem:lk_arctan}) in the limit case $d=0$.
\medskip

The symmetry of the $\at$ function in $a,b$, i.e. $\at(a,b;d,\al)=\at(b,a;d,\al)$ implies that $\lk(L_1,L_2)=\lk(L_2,L_1)$ by~\Cref{theorem:lk_arctan}.
Since the $\at$ function is odd in $d$, i.e. $\at(a,b;-d,\al)=-\at(b,a;d,\al)$,~\Cref{lemma:central_symmetry} is also respected.
\medskip
 
\begin{proof}[of~\Cref{cor:lk_orthogonal_at0}]
By definition any simple orthogonal line segments $L_1,L_2$ have the angle $\al=\dfrac{\pi}{2}$ and initial endpoints $a_1=0=a_2$, hence $b_1=l_1$, $b_2=l_2$.
Substituting the values above into~(\ref{theorem:lk_arctan}) gives
 $\at(0,l_2;d,\frac{\pi}{2})=0$, $\at(l_1,0;d,\frac{\pi}{2})=0$, $\at(0,0;d,\frac{\pi}{2})=0$.
Then $\lk(L_1,L_2)=-\dfrac{1}{4\pi}\at(l_1,l_2;d,\al)=-\dfrac{1}{4\pi}\arctan\left(\dfrac{l_1 l_2}{d\sqrt{l_1^2+l_2^2+d^2}}\right)$.
\bs
\end{proof}

\Cref{fig:AT} shows how the function $\at(a,b;d,\al)$ from~\Cref{theorem:lk_arctan} depends on 2 of 4 parameters when others are fixed.
For example, if $\al=\dfrac{\pi}{2}$, then $\at(a,b;d,\frac{\pi}{2})=\arctan\left(\dfrac{ab}{d\sqrt{a^2+b^2+d^2}}\right)$.
If also $a=b$, the surface $\at(a,a;d,\frac{\pi}{2})=\arctan\left(\dfrac{a^2}{d\sqrt{2a^2+d^2}}\right)$ in the first picture of~\Cref{fig:AT} has the horizontal ridge $\at(0,0;d,\frac{\pi}{2})=0$ and $\lim\limits_{d\to 0}\at(a,a;d,\frac{\pi}{2})=\sign(d)\dfrac{\pi}{2}$ for $a\neq 0$.
%If $l=a-b=0$, then $a^2+b^2-2ab\cos\al =2a^2(1-\cos\al)=4a^2\sin^2\dfrac{\al}{2}$, so $\at(a,a;d,\al)=\arctan\left(\dfrac{a^2\sin\al+d^2\cot\al}{d\sqrt{4a^2\sin^2\frac{\al}{2}+d^2}}\right)$.
%If additionally, if $\al=\dfrac{\pi}{2}$, the first picture in Fig.~\Cref{fig:AT} shows $\at(a,a;d,\frac{\pi}{2})=\arctan\left(\dfrac{a^2}{d\sqrt{2a^2+d^2}}\right)$ with the extremal curves $\at(0,0;d,\frac{\pi}{2})=0$ and . 
If $d,\al$ are free, but $a=0$, then 
$\at(0,0;d,\al)=\arctan\left(\dfrac{d^2\cot\al}{d\sqrt{d^2}}\right)=\sign(d)\arctan(\cot\al)=\sign(d)(\frac{\pi}{2}-\al).$
Similarly, $\lim\limits_{d\to\infty}\at(0,0;d,\al)=\sign(d)(\frac{\pi}{2}-\al)$, see the lines $\at=\frac{\pi}{2}-\al$ on the boundaries of the AT surfaces in the middle pictures of~\Cref{fig:AT}.
\medskip

\sepfigure

\newcommand{\atheight}{38mm}
\begin{figure}[h!]
\includegraphics[height=\atheight]{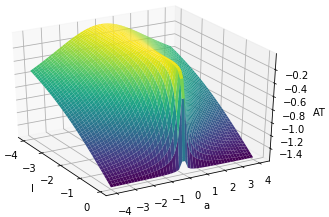}
\includegraphics[height=\atheight]{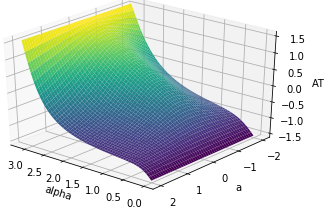}

\includegraphics[height=\atheight]{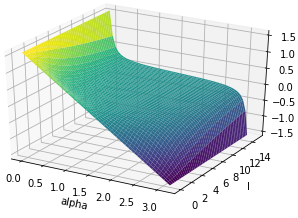}
\hspace*{6mm}
\includegraphics[height=\atheight]{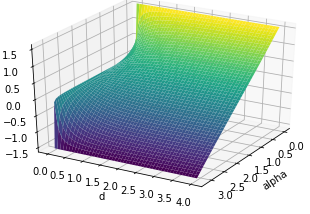}

\includegraphics[height=\atheight]{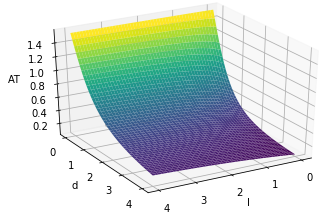}
\includegraphics[height=\atheight]{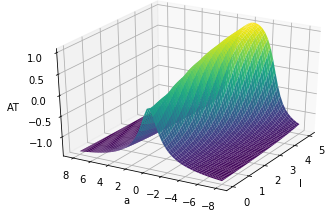}
\caption{The graph of $\at(a,b;d,\al)=\arctan\left(\dfrac{ab\sin\al+d^2\cot\al}{d\sqrt{a^2+b^2-2ab\cos\al+d^2}}\right)$, where 2 of 4 parameters are fixed.
\textbf{Top left}: $l=b-a=0$, $\al=\dfrac{\pi}{2}$.
\textbf{Top right}: $l=d=-1$.
\textbf{Middle left}: $a=0$, $d=1$.
\textbf{Middle right}: $a=0$, $l=1$.
\textbf{Bottom left}: $a=1$, $\al=\dfrac{\pi}{2}$.
\textbf{Bottom right}: $d=-1$, $\al=\dfrac{\pi}{2}$.
}
\label{fig:AT}
\end{figure}

\sepfigure

\begin{lem}[$\lk(L_1,L_2)$ is an integral in $p,q$]
\label{lemma:lk=integral_pq}
In the notations of~\Cref{definition:segments_parameters} we have
$\lk(L_1,L_2)=-\dfrac{1}{4\pi} 
\int\limits_{a_1/d}^{b_1/d} \;
\int\limits_{a_2/d}^{b_2/d}
\dfrac{\sin\al\; dp\; dq}{(1+p^2+q^2-2pq\cos\al)^{3/2}}$ for $d>0$.
\end{lem}

\begin{proof}[\Cref{lemma:lk=integral_pq}]
The following computations assume that $a_1,a_2,l_1,l_2,\al$ are given and $t,s\in[0,1]$. 

\noindent
$L_1(t) = ( (a_1+l_1 t)\cos\frac{\al}{2}, -(a_1+l_1 t)\sin\al,-\frac{d}{2})$, \\ 
$L_2(s) = ( (a_2+l_2 s)\cos\frac{\al}{2}, (a_2+l_2 s)\sin\al,\frac{d}{2})$, \\
$\dot L_1(t)=(l_1\cos\frac{\al}{2}, -l_1\sin\frac{\al}{2},0)$, \\ 
$\dot L_2(s)=(l_2\cos\frac{\al}{2}, l_2\sin\frac{\al}{2},0)$, \\
$\dot L_1(t)\times\dot L_2(s)=(0,0,2l_1 l_2\sin\frac{\al}{2}\cos\frac{\al}{2})=(0,0,l_1l_2\sin\al)$, \\
$L_1(t)-L_2(s) = ( (a_1-a_2+l_1t-l_2s)\cos\al, -(a_1+a_2+l_1 t+l_2s)\sin\al,-d)$, \\
$(\dot L_1(t),\dot L_2(s),L_1(t)-L_2(s))=-dl_1 l_2\sin\al$, \\
\medskip

\noindent
$\lk(L_1,L_2)=\dfrac{1}{4\pi}\int\limits_{0}^1\int\limits_{0}^1 \dfrac{(\dot L_1(t),\dot L_2(s),L_1(t)-L_2(s))}{|L_1(t)-L_2(s)|^3}dt ds=$ \\
$=\dfrac{1}{4\pi}\int\limits_{0}^1\int\limits_{0}^1 \dfrac{-dl_1 l_2\sin\al\; dt ds}{(d^2+(a_1-a_2+l_1t-l_2s)^2\cos^2\frac{\al}{2}+(a_1+a_2+l_1 t+l_2s)^2\sin^2\frac{\al}{2})^{3/2}}$ \\
$=-\dfrac{dl_1 l_2\sin\al}{4\pi}\int\limits_{0}^1\int\limits_{0}^1 \dfrac{dt ds}{(d^2+(a_1-a_2+l_1t-l_2s)^2\cos^2\frac{\al}{2}+(a_1+a_2+l_1 t+l_2s)^2\sin^2\frac{\al}{2})^{3/2}}$.
\medskip

To simplify the last integral, introduce the variables $p=\dfrac{a_1+l_1 t}{d}$ and $q=\dfrac{a_2+l_2 s}{d}$.
In the new variables $p,q$ the expression under the power $\dfrac{3}{2}$ in the denominator becomes
$$d^2+(pd-qd)^2\cos^2\frac{\al}{2}+(pd+qd)^2\sin^2\frac{\al}{2}=$$
$$=d^2\left(1+(p^2-2pq+q^2)\cos^2\frac{\al}{2}+(p^2-2pq+q^2)\cos^2\frac{\al}{2} \right)
=$$
$$=d^2\left(1+p^2\Big(\cos^2\frac{\al}{2}+\sin^2\frac{\al}{2}\Big)+q^2-2pq\Big(\cos^2\frac{\al}{2}-\sin^2\frac{\al}{2}\Big)\right)
=d^2(1+p^2+q^2-2pq\cos\al).$$ 
The old variables are expressed as $t=\dfrac{pd-a_1}{l_1}$, $ts=\dfrac{pd-a_2}{l_2}$ and have the differentials $dt=\dfrac{d}{l_1}dp$, $ds=\dfrac{d}{l_2}dq$.
Since $t,s\in[0,1]$, the new variables $p,q$ have the ranges $[\frac{a_1}{d},\frac{b_1}{d}]$ and $[\frac{a_2}{d},\frac{b_2}{d}]$, respectively.
Then the linking number has the required expression in the lemmama:
$$\lk(L_1,L_2)=-\dfrac{dl_1 l_2\sin\al}{4\pi}
\int\limits_{a_1/d}^{b_1/d} \;
\int\limits_{a_2/d}^{b_2/d}
\dfrac{d^2}{l_1l_2}\dfrac{dp\; dq}{d^3(1+p^2+q^2-2pq\cos\al)^{3/2}}=$$
$$=-\dfrac{1}{4\pi} 
\int\limits_{a_1/d}^{b_1/d} \;
\int\limits_{a_2/d}^{b_2/d}
\dfrac{\sin\al\; dp\; dq}{(1+p^2+q^2-2pq\cos\al)^{3/2}}.$$
Due to~\Cref{lemma:central_symmetry}, the above computations assume that the signed distance $d>0$.
\bs
\end{proof}

\begin{lem}[the linking number as a single integral]
\label{lemma:lk=integral_p}
In the notations of~\Cref{definition:segments_parameters} we have
$\lk(L_1,L_2)=\dfrac{I(a_2/d)-I(b_2/d)}{4\pi}$, where the function $I(r)$ is defined as the single integral
$I(r)=\int\limits_{a_1/d}^{b_1/d} 
\dfrac{\sin\al(r-p\cos\al) dp}{(1+p^2\sin^2\al)\sqrt{1+p^2+r^2-2pr\cos\al}}$ for $d>0$.
\end{lem}

\begin{proof}[\Cref{lemma:lk=integral_p}]
Complete the square in the expression under power $\dfrac{3}{2}$ in~\Cref{lemma:lk=integral_pq}:
$$1+p^2+q^2-2pq\cos\al=1+p^2\sin^2\al+(q-p\cos\al)^2.$$
The substitution $(q-p\cos\al)=(1+p^2\sin^2\al)\tan^2\psi$ for the new variable $\psi$ simplifies the sum of squares to $1+\tan^2\psi=\dfrac{1}{\cos^2\psi}$.
Since $q$ varies within $[\frac{a_2}{d},\frac{b_2}{d}]$, for any fixed $p\in[\frac{a_1}{d},\frac{b_1}{d}]$, the range $[\psi_0,\psi_1]$ of $\psi$ satisfies $\tan\psi_0=\dfrac{\frac{a_2}{d}-p\cos\al}{\sqrt{1+p^2\sin^2\al}}$ and $\tan\psi_1=\dfrac{\frac{b_2}{d}-p\cos\al}{\sqrt{1+p^2\sin^2\al}}$.
Since we treat $p,\psi$ as independent variables, the Jacobian of the substitution $(p,q)\mapsto(p,\psi)$ equals 
$$\dfrac{\bd q}{\bd\psi}=\dfrac{\bd}{\bd\psi}\left(p\cos\al+\tan\psi\sqrt{1+p^2\sin^2\al}\right)=\dfrac{\sqrt{1+p^2\sin^2\al}}{\cos^2\psi}.$$
In the variables $p,\psi$ the expression under the double integral of~\Cref{lemma:lk=integral_pq} becomes
$$
\dfrac{\sin\al\; dp\; dq}{(1+p^2+q^2-2pq\cos\al)^{3/2}}
=\dfrac{\sin\al\; dp}{((1+p^2\sin^2\al)+(1+p^2\sin^2\al)\tan^2\psi)^{3/2}}\dfrac{\bd q}{\bd\psi}d\psi$$
$$=\dfrac{\sin\al\; dp}{(1+p^2\sin^2\al)^{3/2}(1+\tan^2\psi)^{3/2}}\dfrac{d\psi\sqrt{1+p^2\sin^2\al}}{\cos^2\psi}
=\dfrac{\sin\al\; dp \cos\psi\; d\psi}{1+p^2\sin^2\al}
.$$
$$\lk(L_1,L_2)=-\dfrac{1}{4\pi}
\int\limits_{a_1/d}^{b_1/d} 
\dfrac{\sin\al\; dp }{1+p^2\sin^2\al}
\int\limits_{\psi_0}^{\psi_1} \cos\psi\; d\psi
=\dfrac{1}{4\pi}
\int\limits_{a_1/d}^{b_1/d} 
\dfrac{\sin\al\; dp }{1+p^2\sin^2\al}(\sin\psi_0-\sin\psi_1).$$
Express the sin functions for the bounds $\psi_0,\psi_1$ in terms of tan
as $\sin\psi_0=\dfrac{\tan\psi_0}{\sqrt{1+\tan^2\psi_0}}$.
Using $\tan\psi_0=\dfrac{\frac{a_2}{d}-p\cos\al}{\sqrt{1+p^2\sin^2\al}}$ obtained above, we get
$$\sqrt{1+\tan^2\psi_0}=\sqrt{\dfrac{(1+p^2\sin^2\al)+(\frac{a_2}{d}-p\cos\al)^2}{1+p^2\sin^2\al}}
=\sqrt{\dfrac{1+p^2+(\frac{a_2}{d})^2-2\frac{a_2}{d}p\cos\al}{1+p^2\sin^2\al}}.$$
$$\sin\psi_0=\dfrac{\frac{a_2}{d}-p\cos\al}{\sqrt{1+p^2\sin^2\al}}
\sqrt{\dfrac{1+p^2\sin^2\al}{1+p^2+(\frac{a_2}{d})^2-2\frac{a_2}{d}p\cos\al}}=\dfrac{\frac{a_2}{d}-p\cos\al}{\sqrt{1+p^2+(\frac{a_2}{d})^2-2\frac{a_2}{d}p\cos\al}}.$$
Then $\sin\psi_1$ has the same expression with $a_2$ replaced by $b_2$.
After substituting these expressions in the previous formula for the linking number, we get
$\lk(L_1,L_2)=$
$$\dfrac{1}{4\pi}\int\limits_{a_1/d}^{b_1/d} 
\dfrac{\sin\al\; dp}{1+p^2\sin^2\al}\left(
\dfrac{\frac{a_2}{d}-p\cos\al}{\sqrt{1+p^2+(\frac{a_2}{d})^2-2\frac{a_2}{d}p\cos\al}}
-\dfrac{\frac{b_2}{d}-p\cos\al}{\sqrt{1+p^2+(\frac{b_2}{d})^2-2\frac{b_2}{d}p\cos\al}}\right)$$
$=\dfrac{S(a_2/d)-S(b_2/d)}{4\pi}$, where 
$I(r)=\int\limits_{a_1/d}^{b_1/d} 
\dfrac{\sin\al(r-p\cos\al) dp}{(1+p^2\sin^2\al)\sqrt{1+p^2+r^2-2pr\cos\al}}$.
\bs
\end{proof}

\begin{lem}[$I(r)$ via arctan]
\label{lemma:arctan}
The integral $I(r)$ in~\Cref{lemma:lk=integral_p} can be found as
$$\int\dfrac{\sin\al(r-p\cos\al) dp}{(1+p^2\sin^2\al)\sqrt{1+p^2+r^2-2pr\cos\al}}
=\arctan\dfrac{pr\sin\al+\cot\al}{\sqrt{1+p^2+r^2-2pr\cos\al}}
+C.$$
\end{lem}
\medskip

\begin{proof}%[\Cref{lemma:arctan}]
The easy way is to differentiate $\arctan\om$ for 
$\om=\dfrac{pr\sin^2\al+\cos\al}{\sin\al\sqrt{1+p^2+r^2-2pr\cos\al}}$ with respect to the variable $p$ remembering that $r,\al$ are fixed  parameters.
For notational clarity, we use an auxiliary symbol for the expression under the square root: $R=1+p^2+r^2-2pr\cos\al$.
Then $\om=\dfrac{pr\sin^2\al+\cos\al}{\sin\al\sqrt{R}}$ and 
$$\dfrac{d\om}{dp}=\dfrac{1}{R\sin\al}\left(r\sin^2\al\sqrt{R}-(rp\sin^2\al+\cos\al)\dfrac{2p-2r\cos\al}{2\sqrt{R}}\right)=$$
$$=\dfrac{1}{R\sqrt{R}\sin\al}\Big(r\sin^2\al(1+p^2+r^2-2pr\cos\al)
-(rp\sin^2\al+\cos\al)(p-r\cos\al)\Big)=$$
$$\dfrac{rp^2\sin^2\al+r^3\sin^2\al-2pr^2\cos\al\sin^2\al-rp^2\sin^2\al+pr^2\cos\al\sin^2\al-p\cos\al+r}{R\sqrt{R}\sin\al}$$
$$=\dfrac{r^3\sin^2\al-pr^2\cos\al\sin^2\al-p\cos\al+r}{R\sqrt{R}\sin\al}
=\dfrac{(r-p\cos\al)(1+r^2\sin^2\al)}{R\sqrt{R}\sin\al}
.$$
$$\dfrac{d}{dp}\arctan\om=\dfrac{1}{1+\om^2}\cdot\dfrac{d\om}{dp}
=\dfrac{(\sin\al\sqrt{R})^2}{(\sin\al\sqrt{R})^2+(pr\sin^2\al+\cos\al)^2}\cdot\dfrac{d\om}{dp}=$$
$$=\dfrac{R\sin^2\al}{R\sin^2\al+(p^2r^2\sin^4\al+2pr\sin^2\al\cos\al+\cos^2\al)}\cdot
\dfrac{(r-p\cos\al)(1+r^2\sin^2\al)}{R\sqrt{R}\sin\al}=$$
$$=\dfrac{\sin\al}{\sqrt{R}}\cdot\dfrac{(r-p\cos\al)(1+r^2\sin^2\al)}{\sin^2\al(1+p^2+r^2-2pr\cos\al)+(p^2r^2\sin^4\al+2pr\sin^2\al\cos\al+\cos^2\al)}=$$
$$=\dfrac{\sin\al(r-p\cos\al)(1+r^2\sin^2\al)}{(1+p^2\sin^2\al+r^2\sin^2\al+p^2r^2\sin^4\al)\sqrt{R}}
=\dfrac{\sin\al(r-p\cos\al)(1+r^2\sin^2\al)}{(1+p^2\sin^2\al)(1+r^2\sin^2\al)\sqrt{R}}=$$
$$
=\dfrac{\sin\al(r-p\cos\al)}{(1+p^2\sin^2\al)\sqrt{R}}
=\dfrac{\sin\al(r-p\cos\al)}{(1+p^2\sin^2\al)\sqrt{1+p^2+q^2-2pq\cos\al}}.$$
Since we got the required expression under the integral $I(r)$,~\Cref{lemma:arctan} is proved.
\bs
\end{proof}

\begin{proof}[\Cref{theorem:lk_arctan}]
Consider the right hand side of the equation in~\Cref{lemma:arctan} as the 3-variable function
$F(p,r;\al)=\arctan\left(\dfrac{pr\sin^2\al+\cos\al}{\sqrt{1+p^2+r^2-2pr\cos\al}}\right)$.
The function in~\Cref{lemma:lk=integral_p} is
$I(r)=F(b_1/d,r;\al)-F(a_1/d,r;\al)$.
By~\Cref{lemma:lk=integral_p} $\lk(L_1,L_2)
=$
$$\dfrac{\Big(F(b_1/d,a_2/d;\al)-F(a_1/d,a_2/d;\al)\Big)-\Big(F(b_1/d,b_2/d;\al)-F(a_1/d,b_2/d;\al)\Big)}{4\pi}.$$
Rewrite a typical function from the numerator above as follows:
$F\left(a/d,b/d;\al\right)=$
$$\arctan\dfrac{(ab/d^2)\sin^2\al+\cos\al}{\sqrt{1+(a/d)^2+(b/d)^2-2(ab/d^2)\cos\al}}
=\arctan\dfrac{ab\sin\al+d^2\cot\al}{d\sqrt{a^2+b^2-2ab\cos\al+d^2}}.$$
If we denote the last expression as $\at(a,b;d,\al)$, required formula~(\ref{theorem:lk_arctan}) follows.
\medskip

In~\Cref{lemma:lk=integral_pq}, \Cref{lemma:lk=integral_p} and above we have used that the signed distance $d$ is positive.
By~\Cref{lemma:central_symmetry} the signed distance $d$ and $\lk(L_1,L_2)$ simultaneously change their signs under a central symmetry, while all other invariants remain the same.
Since $\at(a,b;-d,\al)=-\at(a,b;d,\al)$ due to the arcsin function being odd, formula~(\ref{theorem:lk_arctan}) holds for $d<0$.
The formula remains valid even for $d=0$, when $L_1,L_2$ are in the same plane.
The expected value $\lk(L_1,L_2)=0$ needs an explicit setting, see the discussion of the linking number discontinuity around $d=0$ in~\Cref{cor:lk_d->0}.
\bs
\end{proof}

%6=================
\section{The asymptotic behaviour of the linking number of segments}
\label{sec:behaviour}

This section discusses how the linking number $\lk(L_1,L_2)$ in~\Cref{theorem:lk_arctan} behaves with respect to the six parameters of line segments $L_1,L_2$.
~\Cref{fig:lk} shows how the linking number between two equal line segments varies with different pairs of parameters.  
%For example, if
\smallskip

\sepfigure

\begin{figure}[h!]
\includegraphics[height=\atheight]{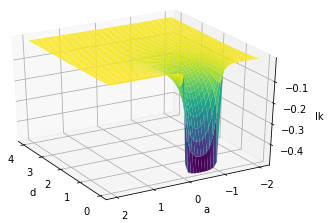}
\includegraphics[height=\atheight]{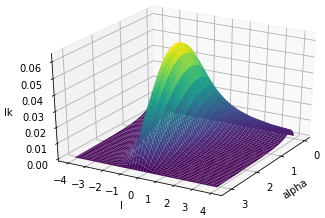}

\includegraphics[height=\atheight]{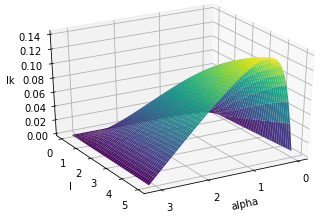}
\hspace*{4mm}
\includegraphics[height=\atheight]{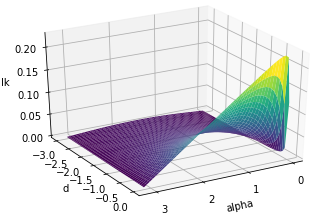}

\includegraphics[height=\atheight]{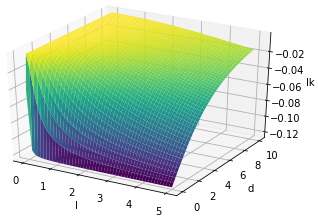}
\hspace*{6mm}
\includegraphics[height=\atheight]{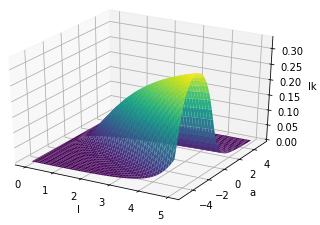}
\caption{The linking number $\lk(a,a+l;a,a+l;d,\al)$ from formula (\ref{theorem:lk_arctan}), where 2 of 4 parameters are fixed.
\textbf{Top left}: $l=1$, $\al=\dfrac{\pi}{2}$.
\textbf{Top right}: $l=1,d=-1$.
\textbf{Middle left}: $a=0$, $d=1$.
\textbf{Middle right}: $a=0$, $l=1$.
\textbf{Bottom left}: $a=0$, $\al=\dfrac{\pi}{2}$.
\textbf{Bottom right}: $d=-1$, $\al=\dfrac{\pi}{2}$.
}
\label{fig:lk}
\end{figure}

\sepfigure

\begin{crl}[bounds of the linking number]
\label{cor:lk_bounds}
%For any line segments $L_1,L_2$, 
For any line segments $L_1,L_2\subset\R^3$, the linking number $\lk(L_1,L_2)$ is between $\pm\dfrac{1}{2}$.
% and approaches these bounds only when $d\to 0$ as described in Corollary~\Cref{cor:lk_d->0}.
%\bs
\end{crl}

\begin{proof}[\Cref{cor:lk_bounds}]
By~\Cref{theorem:lk_arctan} $\lk(L_1,L_2)$ is a sum of 4 arctan functions divided by $4\pi$.
Since each arctan is strictly between $\pm\dfrac{\pi}{2}$, the linking number is between $\pm\dfrac{1}{2}$.
%If $\lk(L_1,L_2)$ approaches one of these bounds, the corresponding 4 arctan should approach $\pm\dfrac{\pi}{2}$, hence their arguments tend to $\pm\infty$.
%The typical argument $\dfrac{ab\sin\al+d^2\cot\al}{d\sqrt{a^2+b^2-2ab\cos\al+d^2}}$ of each arctan approaches $\pm\infty$ only when the denominator tends to 0, hence $d\to 0$.
\bs
\end{proof}

\begin{crl}[sign of the linking number]
\label{cor:lk_sign}
In the notation of~\Cref{definition:segments_parameters}, %$\lk(L_1,L_2)=0$ for $\al=0$ or $\al=\pi$, also 
we have
$\lim\limits_{\al\to 0} \lk(L_1,L_2)=0=\lim\limits_{\al\to\pi} \lk(L_1,L_2)$.
Any non-parallel $L_1,L_2$ have $\sign(\lk(L_1,L_2))=-\sign(d)$.
So $\lk(L_1,L_2)=0$ if and only if $d=0$ or $\al=0$ or $\al=\pi$.
%\bs
\end{crl}

\begin{proof}[\Cref{cor:lk_sign}]
If $\al=0$ or $\al=\pi$, then $\cot\al$ is undefined, so~\Cref{theorem:lk_arctan} sets $\at(a,b;d,\al)=\sign(d)\dfrac{\pi}{2}$.
Then $\lk(L_1,L_2)=\sign(d)\dfrac{\pi}{2}(1+1-1-1)=0$. 
\smallskip

~\Cref{theorem:lk_arctan} also specifies that $\lk(L_1,L_2)=0$ for $d=0$.
If $d\neq 0$ and $\al\to 0$ within $[0,\pi]$ while all other parameters remain fixed, then $d^2\cot\al\to+\infty$.
Hence each of the 4 arctan functions in~\Cref{theorem:lk_arctan} approaches $\dfrac{\pi}{2}$, so $\lk(L_1,L_2)\to 0$.
The same conclusion similarly follows in the case $\al\to\pi$ when $d^2\cot\al\to-\infty$.
\smallskip

If $L_1,L_2$ are not parallel, the angle $\al$ between them belongs to $(0,\pi)$. 
In $d>0$,~\Cref{lemma:lk=integral_pq} says that
$\lk(L_1,L_2)=-\dfrac{1}{4\pi} 
\int\limits_{a_1/d}^{b_1/d} \;
\int\limits_{a_2/d}^{b_2/d}
\dfrac{\sin\al\; dp\; dq}{(1+p^2+q^2-2pq\cos\al)^{3/2}}$.
Since the function under the integral is strictly positive, $\lk(L_1,L_2)<0$.
By~\Cref{lemma:central_symmetry} both $\lk(L_1,L_2)$ simultaneously change their signs under a central symmetry.
Hence the formula $\sign(\lk(L_1,L_2))=-\sign(d)$ holds for all $d$ including $d=0$ above.
\bs
\end{proof}

\begin{crl}[lk for $d\to 0$]
\label{cor:lk_d->0}
If $d\to 0$ and $L_1,L_2$ remain disjoint, formula~(\ref{theorem:lk_arctan}) is continuous and $\lim\limits_{d\to 0} \lk(L_1,L_2)=0$.
If $d\to 0$ and $L_1,L_2$ intersect each other in the limit case $d=0$, then $\lim\limits_{d\to 0} \lk(L_1,L_2)=-\dfrac{\sign(d)}{2}$, where $d\to 0$ keeps its sign.
%\bs
\end{crl}

\begin{proof}[\Cref{cor:lk_d->0}]
Recall that $\lim\limits_{x\to\pm\infty}\arctan x=\pm\dfrac{\pi}{2}$.
By~\Cref{cor:lk_sign} assume that $\al\neq 0,\al\neq\pi$, so $\al\in(0,\pi)$.
Then $\sin\al>0$, $a^2+b^2-2ab\cos\al>(a-b)^2\geq 0$ and
$$\lim\limits_{d\to 0}\at(a,b;d,\al)=\lim\limits_{d\to 0}\arctan\left(\dfrac{ab\sin\al+d^2\cot\al}{d\sqrt{a^2+b^2-2ab\cos\al+d^2}}\right)=$$ 
$$=\sign(a)\sign(b)\sign(d)\dfrac{\pi}{2},\text{so \Cref{theorem:lk_arctan} gives }$$
$$\lim\limits_{d\to 0}\lk(L_1,L_2)=\dfrac{\sign(d)}{8}(\sign(a_1)-\sign(b_1))(\sign(b_2)-\sign(a_2)).$$
In the limit case $d=0$, the line segments $L_1,L_2\subset\{z=0\}$ remain disjoint in the same plane if and only if both endpoint coordinates $a_i,b_i$ have the same sign for at least one of $i=1,2$, which is equivalent to $\sign(a_i)-\sign(b_i)=0$, i.e. $\lim\limits_{d\to 0}\lk(L_1,L_2)=0$ from the product above.
Hence formula~(\ref{theorem:lk_arctan}) is continuous under $d\to 0$ for any non-crossing segments.
Any segments that intersect in the plane $\{z=0\}$ when $d=0$ have endpoint coordinates $a_i<0<b_i$ for both $i=1,2$ and have the limit
$\lim\limits_{d\to 0}\lk(L_1,L_2)=\dfrac{\sign(d)}{8}(-1-1)(1-(-1))=-\dfrac{\sign(d)}{2}$ as required. 
\bs
\end{proof}

\begin{crl}[lk for $d\to\pm\infty$]
\label{cor:lk_d->infty}
If the distance $d\to\pm\infty$, then $\lk(L_1,L_2)\to 0$.
%\bs
\end{crl}

\begin{proof}[\Cref{cor:lk_d->infty}]
If $d\to\pm\infty$, while other parameters of $L_1,L_2$ remain fixed, then the function $\at(a,b;d,\al)=\arctan\left(\dfrac{ab\sin\al+d^2\cot\al}{d\sqrt{a^2+b^2-2ab\cos\al+d^2}}\right)$ from~\Cref{theorem:lk_arctan}
has the limit $\arctan(\sign(d)\cot\al)=\sign(d)\left(\dfrac{\pi}{2}-\al\right)$.
Since the four AT functions in~\Cref{theorem:lk_arctan} include the same $d,\al$, their limits cancel, so $\lk(L_1,L_2)\to 0$.
\bs
\end{proof}

\begin{crl}[lk for $a_i,b_i\to\infty$]
\label{cor:lk_ab->infty}
If the invariants $d,\al$ of line segments $L_1,L_2\subset\R^3$ remain fixed, but $a_i\to+\infty$ or $b_i\to-\infty$ for each $i=1,2$, then $\lk(L_1,L_2)\to 0$.
%\bs
\end{crl}

\begin{proof}[\Cref{cor:lk_ab->infty}]
If $a_i\to+\infty$, then $a_i\leq b_i\to+\infty$, $i=1,2$.
If $b_i\to-\infty$, then $b_i\geq a_i\to-\infty$, $i=1,2$.
Consider the former case $a_i\to+\infty$, the latter is similar.
\medskip

Since $d,\al$ are fixed, $a^2+b^2-2ab\cos\al+d^2\leq (a+b)^2+d^2\leq 5b^2$ for large enough $b$.
Since $\arctan(x)$ increases, $\at(a,b;d,\al)\geq\arctan\left( \dfrac{ab\sin\al+d^2\cot\al}{db\sqrt{5}}\right)\to\sign(d)\dfrac{\pi}{2}$ as $b\geq a\to+\infty$.
Since the four AT functions in~\Cref{theorem:lk_arctan} have the same limit when their first two arguments tend to $+\infty$, these 4 limits cancel, so $\lk(L_1,L_2)\to 0$.
\bs
\end{proof}

\begin{crl}[lk for $a_i\to b_i$]
\label{cor:lk_l->0}
If one of segments $L_1,L_2\subset\R^3$ becomes infinitely short so that its final endpoint tends to the fixed initial endpoint (or vice versa), while all other invariants of $L_1,L_2$ from~\Cref{definition:segments_parameters} remain fixed, then $\lk(L_1,L_2)\to 0$. 
%\bs
\end{crl}

\begin{proof}[\Cref{cor:lk_l->0}] $\lk(L_1,L_2)=0$ for $d=0$.
It's enough to consider the case $d\neq 0$.
Then $\at(a,b;d,\al)=\arctan\left(\dfrac{ab\sin\al+d^2\cot\al}{d\sqrt{a^2+b^2-2ab\cos\al+d^2}}\right)$ %from~\Cref{theorem:lk_arctan} 
is continuous.  
Let (say for $i=1$) $a_1\to b_1$, the case $b_1\to a_1$ is similar.
The continuity of $\at$ implies that $\at(a_1,b_2;d,\al)\to\at(b_1,b_2;d,\al)$ and $\at(a_1,a_2;d,\al)\to\at(b_1,a_2;d,\al)$.
In the limit all terms in~\Cref{theorem:lk_arctan} cancel, hence $\lk(L_1,L_2)\to 0$.
\bs
\end{proof}

%7==============
\section{Example computations of the linking number and a discussion}
\label{sec:computations}

If curves $\ga_1,\ga_2\subset\R^3$ consist of straight line segments, then $\lk(\ga_1,\ga_2)$ can be computed as the sum of $\lk(L_1,L_2)$ over all line segments $L_1\subset\ga_1$ and $L_2\subset\ga_2$.

\sepfigure
\begin{figure}[h]
\centering
\includegraphics[height=23mm]{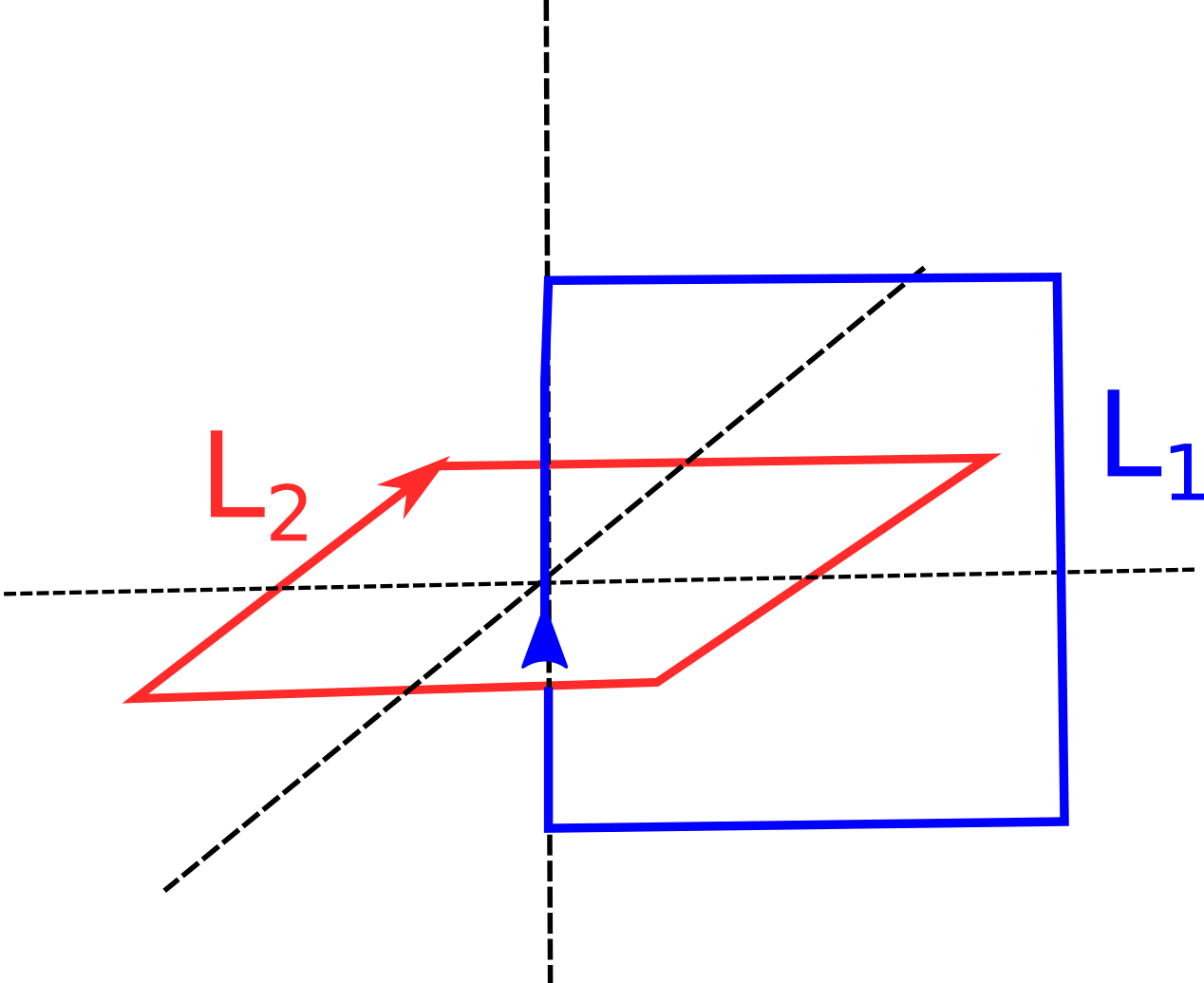}
\hspace*{0mm}
\includegraphics[height=23mm]{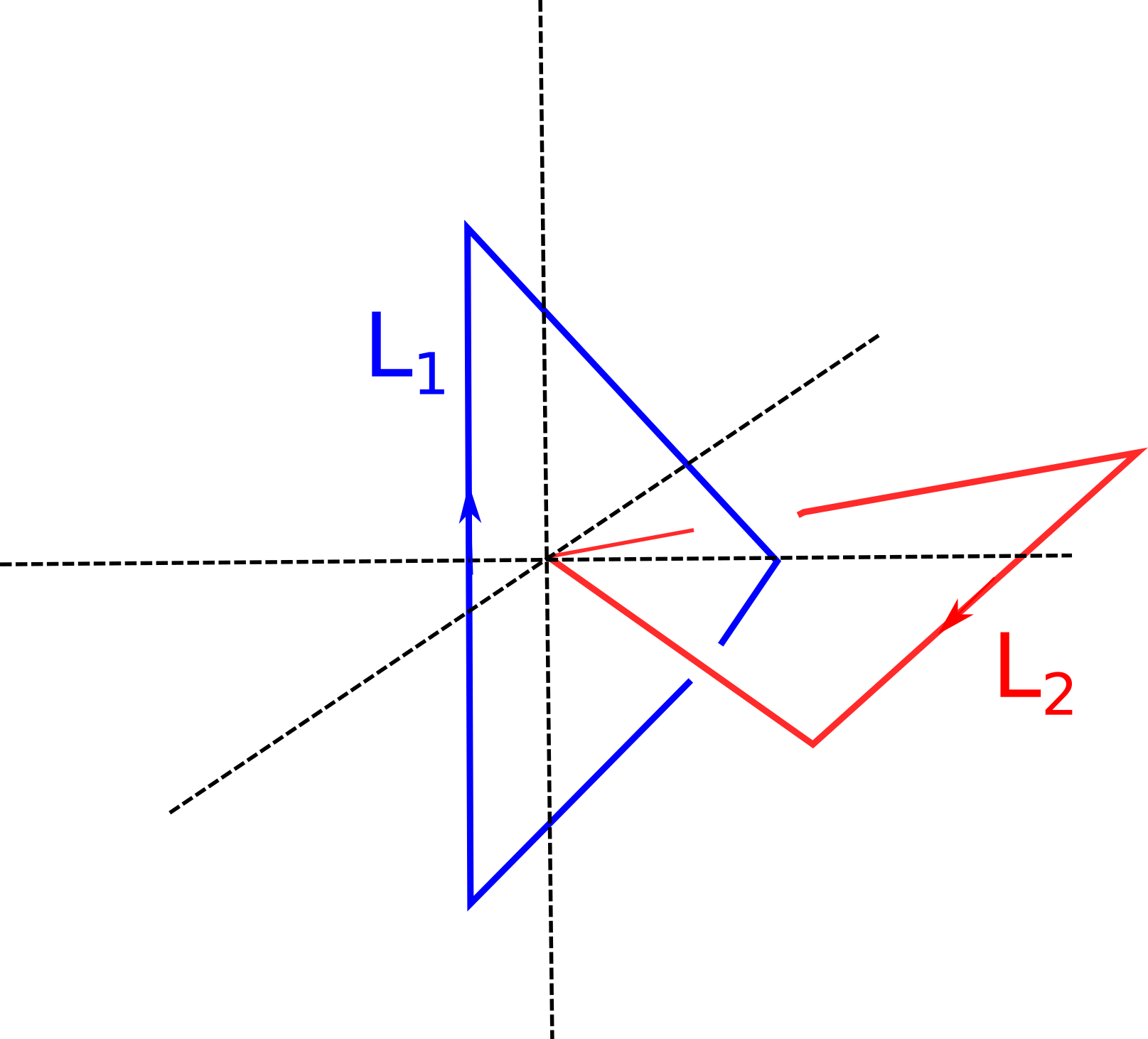}
\includegraphics[height=22mm]{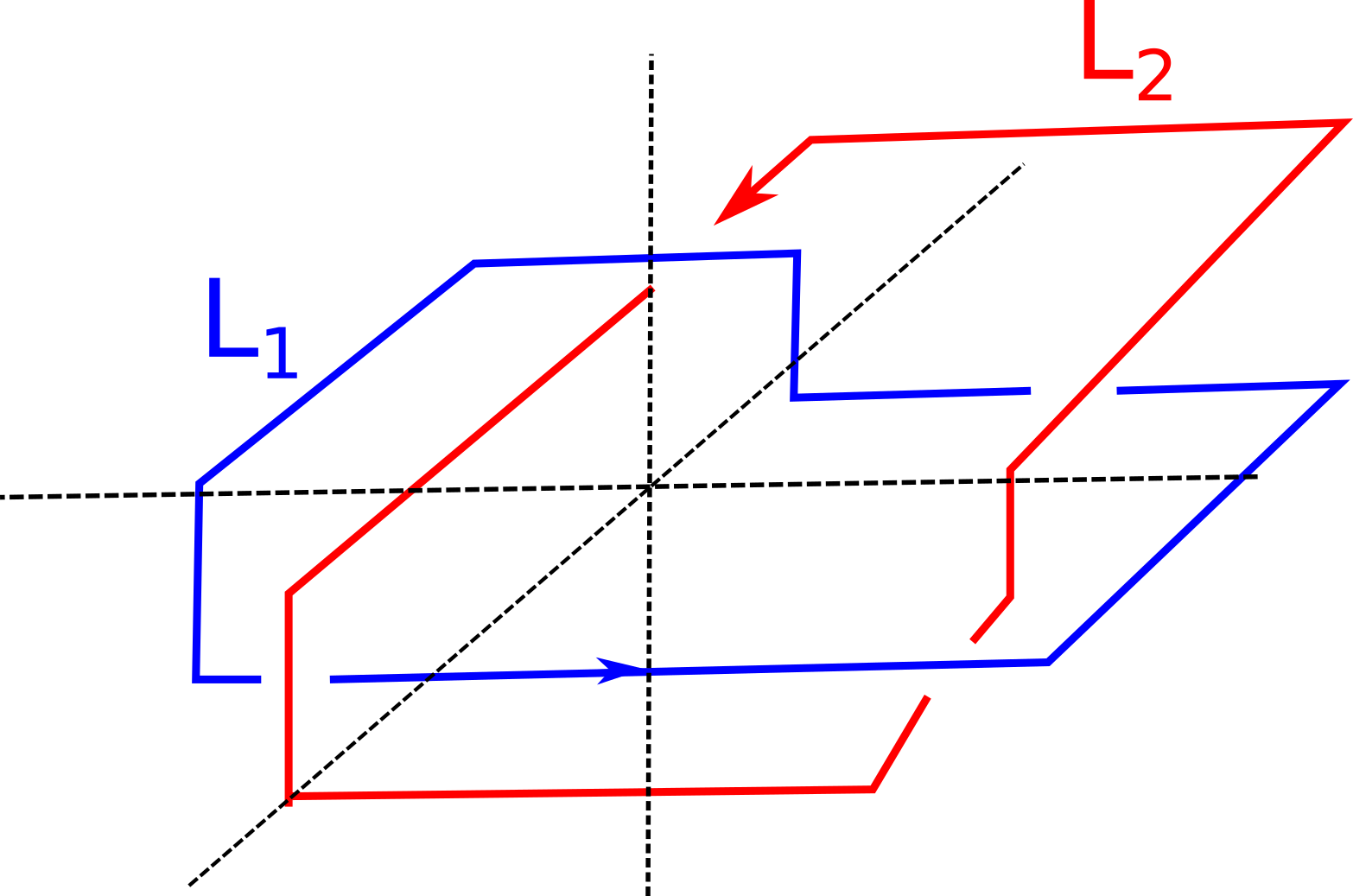}
\hspace*{0mm}
\includegraphics[height=22mm]{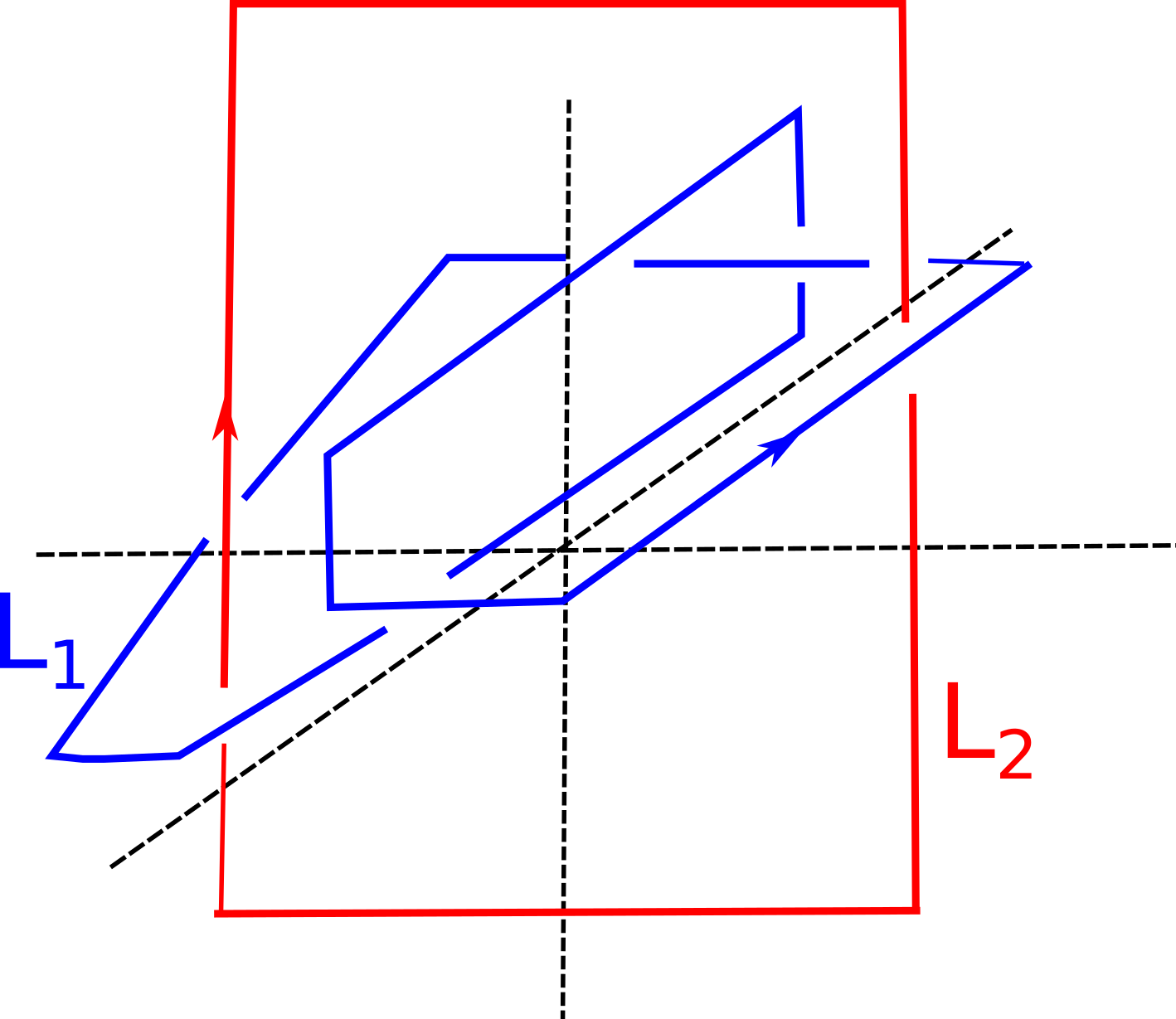}
\caption{\textbf{1st}: The Hopf link as two square cycles has $\lk=-1$ and vertices with coordinates $L_1 = (-2,0,2),(2,0,-2),(2,0,2),(-2,0,2)$, 
$L_2 = (-1,-2,0),(-1,2,0),(1,2,0),(1,-2,0)$
\textbf{2nd}: The Hopf link of triangular cycles has $\lk=+1$, $L_1 = (-1,0,-1),(-1,0,1),(1,0,0)$ and $L_2 = (0,0,0),(2,1,0),(2,-1,0)$.
\textbf{3rd}: Solomon's link has $\lk=+2$, $L_1 = (-1,1,1),(-1,-1,1),(3,-1,1),(3,1,-1),(1,1,-1),(1,1,1)$ and $L_2 = (-1,-2,0),(-1,2,0),(1,2,0),(1,-2,0)$.
\textbf{4th}: Whitehead's link has $\lk=0$, $L_1=$
$(-3,-2,-1),(0,-2,-1),(0,2,1),(0,0,1),(0,0,0),(3,0,0),(3,1,0),(-3,1,0),(-3,1,-1)$ and
$L_2 = (-1,0,-3),(-1,0,3),(1,0,3),(-1,0,3)$.}
\label{fig:links}
\end{figure}
\sepfigure

\Cref{fig:links} shows polygonal links whose linking numbers were computed by our Python code implementing formula~(\ref{theorem:lk_arctan}) at \url{https://github.com/MattB-242/Closed_Lk_Form}.
For all links in  in Fig~\ref{fig:links} formula~\ref{theorem:lk_arctan} calculates the linking number between the two components correctly (as equal to $-1$ and $+1$ respectively in the orientations given in Fig~\ref{fig:links}), with a computation error of less than $10^{-12}$. 
%In~\Cref{fig:otherlinks} we show polygonal versions of the next two most complex links in the Thistlethwaite link table \cite{barnatan} - 'Solomon's Knot' (L4A1) and the Whitehead link (L5A1). In both cases the $lk$ formula calculates the linking number between the two components correctly as $+2$ and $0$ respectively, again with computation error $< 10^{-12}$
\medskip

The asymptotic linking number introduced by Arnold
converges for infinitely long curves \cite{vogel2003asymptotic}, while
our initial motivation was a computation of geometric and topology invariants to classify periodic structures such as textiles \cite{bright2020encoding} and crystals \cite{cui2019mining}. 
\medskip

\Cref{theorem:lk_arctan} allows us to compute the \emph{periodic} linking number between a segment $J$ and a growing finite lattice $L_n$ whose unit cell consists of $n$ copies of two oppositely oriented segments orthogonal to $J$.
This periodic linking number is computed for increasing $n$ in a lattice extending periodically in one, two and three directions, see~\Cref{fig:perilink}. 
As $n$ increases, the $\lk$ function asymptotically approaches an approximate value of $0.30$ for 1- and 3-periodic lattice and $0.29$ for the 2-periodic lattice. 
\medskip

The invariant-based formula has allowed us to prove new asymptotic results of the linking number in Corollaries \ref{cor:lk_bounds}-\ref{cor:lk_l->0} of section~\ref{sec:behaviour}. 
Since the periodic linking number is a real-valued invariant modulo isometries, it can be used to continuously quantify similarities between periodic crystalline networks \cite{cui2019mining}. 
One next possible step is to use formula (\ref{theorem:lk_arctan}) to prove asymptotic convergence of the periodic linking number for arbitrary lattices, so that we can show that the limit of the infinite sum is a general invariant that can be used to develop descriptors of crystal structures.
\medskip

%It is possible that a similar approach to the one taken here may provide a route to more general computation for piecewise linear knots. 
The Milnor invariants generalise the linking number to invariants of links with more than two components. 
An integral for the three component Milnor invariant~\cite{deturck2011} may be possible to compute in a closed form similarly to Theorem~\ref{theorem:lk_arctan}. 
The interesting open problem is to extend the isometry-based approach to finer invariants of knots.
\medskip

%However, there are also applications in the field of knot theory itself. 
The Gauss integral in (\ref{definition:Gauss_integral}) was extended to the infinite Kontsevich integral containing all finite-type or Vassiliev's invariants of knots~\cite{kontsevich1993}.
The coefficients of this infinite series were explicitly described \cite{kurlin2005compressed} as solutions of exponential equations with non-commutative variables $x,y$ in a compressed form modulo commutators of commutators in $x,y$.
The underlying metabelian formula for $\ln(e^x e^y)$ has found an easier proof \cite{kurlin2007baker} in the form of a generating series in the variables $x,y$.
\medskip

In conclusion, we have proved the analytic formula for the linking number based on 6 isometry invariants that uniquely determine a relative position of two line segments in $\R^3$.
Though a similar formula was claimed in \cite{klenin2000computation}, no proof was given.
Hence this paper fills an important gap in the literature by completing the proof via 3 non-trivial lemmas in section~\ref{sec:lk_formula}, see detailed computations in the arxiv version of this paper.

\sepfigure
\begin{figure}[h!]
\centering
\includegraphics[height=55mm]{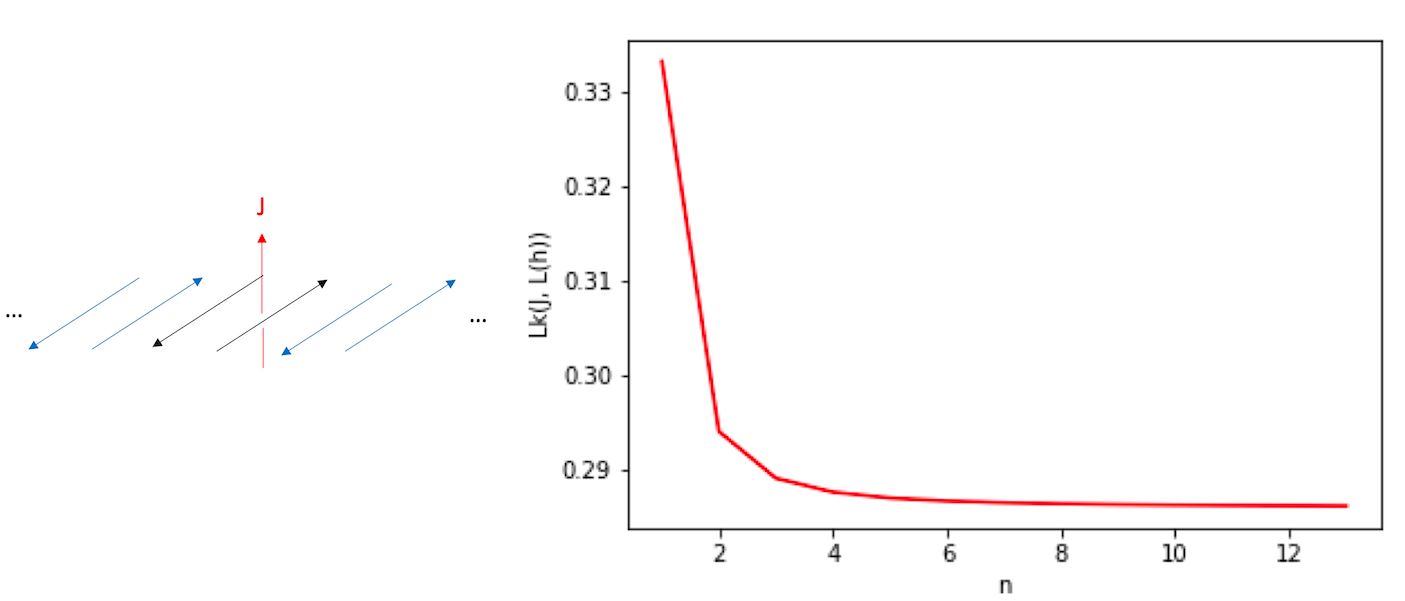}
\hspace*{0mm}
\includegraphics[height=53mm]{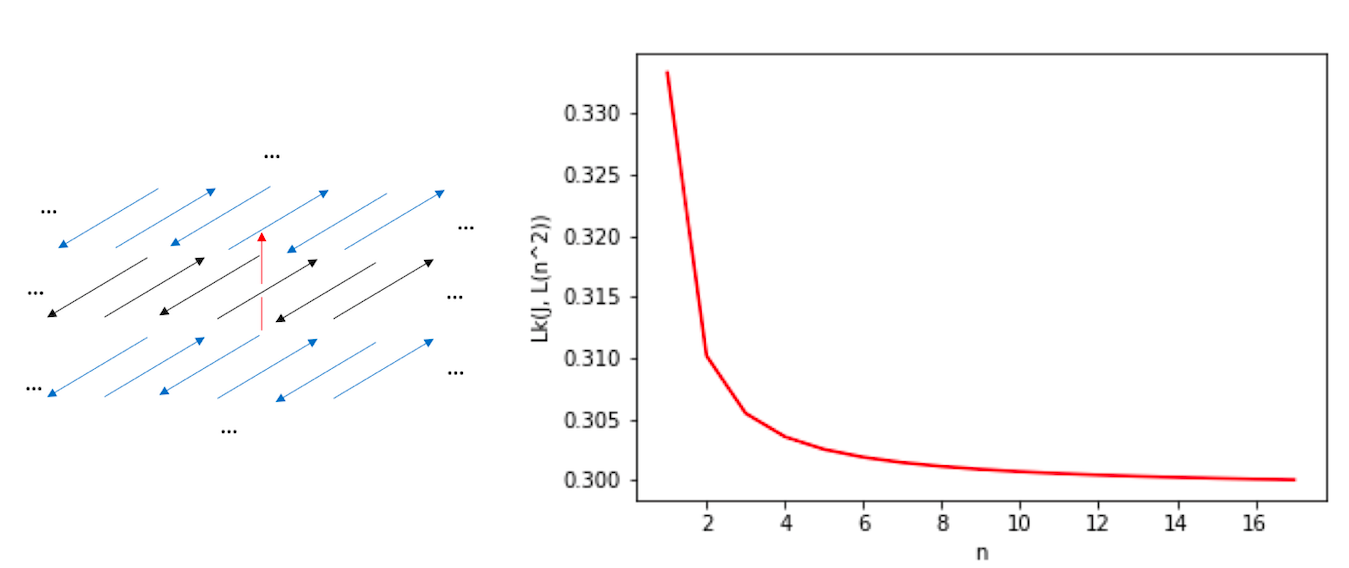}
\hspace*{0mm}
\includegraphics[height=54mm]{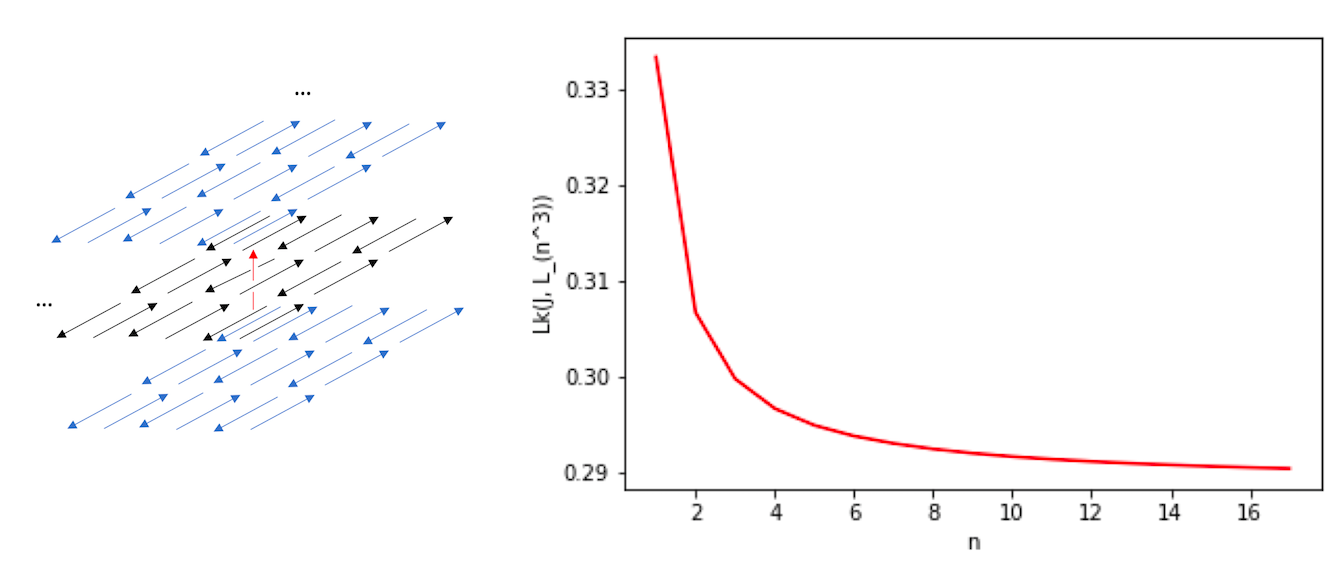}
\caption{\textbf{Left}: the line segment $J = (0,0,-1) + t(0,0,2)$ in red and the periodic lattice $L(n^k)$ derived from $n$  copies of the 'unit cell' $L = \{(-1,-1,0) + t(0,2,0), (-1,1,0) + s(0,-2,0)\}$, $t,s\in[0,1]$, translated in $k$ linearly independent directions for increasing $n\in\Z$.
\textbf{Right}: the periodic linking number $\lk(J,L(n^k))$ is converging fast for $n\to+\infty$.  
\textbf{Top}: $k= 1$. 
\textbf{Middle}: $k=2$. 
\textbf{Bottom}: $k=3$.
}
\label{fig:perilink}
\end{figure}
\sepfigure

\section{Acknowledgements}
This work was supported by the UK Engineering and Physical Sciences Research Council under the grant  ``Application-driven Topological Data Analysis'', EP/R018472/1.
We thank the organisers of NUMGRID2020 for the opportunity to present the results.

\bibliographystyle{plainurl}% the recommended bibstyle
\bibliography{linking-number}

\end{document}